\documentclass{aims_logopdf}

\usepackage{txfonts}
\def\typeofarticle{Research article}
\def\currentvolume{1}
\def\currentissue{1}
\def\currentyear{2018}

\def\ppages{55--83}
\def\DOI{10.3934/Mine.2018.1.55}
\def\Received{19 June 2018}
\def\Accepted{25 July 2018}
\def\Published{14 September 2018}
\emergencystretch=\maxdimen
\numberwithin{equation}{section}

\makeatletter
\renewcommand{\@biblabel}[1]{#1\hfill \hspace{-0.2cm}}
\makeatother

\usepackage{psfrag}
\usepackage[percent]{overpic}

\newcommand{\comments}[1]{}
\newcommand{\bea}{\begin{eqnarray}}
\newcommand{\eda}{\end{eqnarray}}
\newcommand{\bdis}{\begin{displaymath}}
\newcommand{\edis}{\end{displaymath}}
\newcommand{\bee}{\begin{equation}}
\newcommand{\ede}{\end{equation}}
\newcommand{\R}{\mathbb{R}}
\newcommand{\N}{\mathbb{N}}
\newcommand{\vmin}{\underbar{\em{v}}}
\newcommand{\vmax}{\bar{v}}
\newcommand{\rp}{\rho_r}
\renewcommand{\rm}{\rho_\ell}

\DeclareMathOperator{\cco}{\overline{co}}

\newtheorem{remark}{Remark}
\newtheorem{lemma}{Lemma}
\newtheorem{proposition}{Proposition}
\newtheorem{definition}{Definition}

\begin{document}

\title{Two algorithms for a fully coupled and consistently macroscopic PDE-ODE system modeling a moving bottleneck on a road}

\author{%
  Gabriella Bretti\affil{1},
  Emiliano Cristiani\affil{1},
  Corrado Lattanzio\affil{2},
  Amelio Maurizi\affil{2}
  and
  Benedetto Piccoli\affil{3,}\corrauth
}

\shortauthors{the Author(s)}

\address{%
  \addr{\affilnum{1}}{Istituto per le Applicazioni del Calcolo ``M. Picone'', Consiglio Nazionale delle Ricerche, Via dei Taurini 19, 00185 - Rome, Italy}
  \addr{\affilnum{2}}{Dipartimento di Ingegneria e Scienze dell'Informazione e Matematica, Universit\`a degli Studi dell'Aquila, Via Vetoio, 67100 - Coppito (AQ), Italy}
  \addr{\affilnum{3}}{Department of Mathematical Sciences, Rutgers University, \ 311 N 5th Street, Camden, NJ 08102, USA}}

\corraddr{Email: piccoli@camden.rutgers.edu; Tel: 8562256356; Fax: 8562256624.
}

\begin{abstract}
In this paper we propose two numerical algorithms to solve a coupled
PDE-ODE system which models a slow vehicle (bottleneck) moving
on a road together with other cars.
The resulting system is fully coupled because the dynamics of the slow
vehicle depends on the density of cars and, at the same time, it
causes a capacity drop in the road, thus limiting the car flux.
The first algorithm, based on the Wave Front Tracking method,
is suitable for theoretical investigations and convergence results. The second one,
based on the Godunov scheme, is used for numerical simulations.
The case of multiple bottlenecks is also investigated.
\end{abstract}

\keywords{
conservation laws; discontinuous ODEs; fluid dynamic models; LWR model}

\maketitle

\section{Introduction}
In this paper we focus on a system modeling a slow vehicle (e.g.\ large truck or bus) which moves in a car traffic flow. We assume that the slow vehicle's velocity is influenced by the car density and, at the same time, it causes a capacity drop in the road at the position where it is located. Therefore we use the terminology {\em moving bottleneck}.

Mathematically, the problem gives rise to a coupled PDE-ODE system. Car density $\rho$ solves the following conservation law

 \begin{align}\label{lwr}
 \left\{ \begin{array}{lll}
\partial_t \rho(t,x) + \partial_x \Big[ f(\rho(t,x),x-y(t)) \Big]= 0,\\
\rho(0,x)= \bar\rho(x),
  \end{array}
   \right.
 \end{align}
with
$t \in [0,T]$,
$\rho \in [0,\rho_{max}]$,
$(t,x) \in \R^+\times\R$,
$f:[0,\rho_{max}]\times\R\to\R$ the flux function,
$\bar\rho$ the initial datum and
$y(t)$ the position of the moving bottleneck, which is assumed to solve the following Cauchy problem
\bee \label{ode}
 \left\{
 \begin{array}{lll}
 \dot{y}(t) = w(\rho(t,y)),\\
  y(0) = y_0,
  \end{array}
   \right.
\ede
with $y_0$ the position at initial time $t=0$ and $w\geq 0$ the velocity of the moving bottleneck. An example of $f$ and $w$ will be given in the next Section, see equations (\ref{vel_bus})--(\ref{flusso}).

It is well known that the function $\rho(t,\cdot)$ may be discontinuous, then solutions to both (\ref{lwr}) and (\ref{ode}) are to be intended in generalized sense.
The case of multiple moving bottlenecks will be also considered.

\paragraph{Relevant literature.} Let us mention a few results related to our work.
Modelling of vehicular traffic with fluid dynamic approach started with the seminal works of Lighthill and Whitham \cite{WL} and Richards \cite{RI}.
The attention on this subject continued in the last years, both for the proposal of more accurate models, see, e.g., \cite{AwR, Colombo, colombo2016M3AS}, and for the treatment of complex networks, see, e.g., \cite{bretti2014DCDS-S, briani2014NHM, herty, HR} and the book \cite{GPbook}.
To deal with the case of different vehicles, such as motorcycles, cars, buses, multi-population models were also considered, see \cite{benzoni-colombo2003,wong}.

The problem of determining the trajectory of a single car given the density of cars was addressed in \cite{CM}, and the relative numerics in \cite{BP08}, taking advantage of vector techniques. In this case the observed car is assumed not to influence significantly the traffic flow, thus the system is not fully coupled.
Here we consider the more complicated situation in which the position of the single vehicle influences the whole traffic flow.

The coupled system (\ref{lwr})--(\ref{ode}) was proposed in \cite{LMP09}. More precisely, that paper introduced a notion of solution based on weak solutions for the PDE and Filippov solutions for the ODE, and proved existence of such a solution under suitable assumptions.

Recently, another fully coupled system for a moving bottleneck problem was studied in \cite{goatin2014DCDS-S, goatin2014JDE, goatinNHM2017, piccolipreprint2018} from both the theoretical and numerical point of view (see also \cite{goatin2017DCDSSB} for an extension to the second order ARZ model) . Furthermore, in \cite{goatin2018preprint} an optimal control problem using the maximal speed of the coordinated vehicle as control variable is stated. It is worth to point out the difference with the model we consider here: the model in  \cite{goatin2014DCDS-S, goatin2014JDE}, following the lines of \cite{New98}, assumes that the moving bottleneck reduces the maximal density attainable in the road, together with its maximal flux (i.e.\ the capacity of the road). Moreover, it assumes that the moving bottleneck moves at constant speed for low car density. Instead, in the model we consider here the moving bottleneck does not reduce the maximal density attainable in the road, while it only reduces the maximal flux. Moreover, its velocity is always dependent on the car density. These features make the model fully consistent with a pure macroscopic point of view.

Let us also recall that the problem of modeling bottlenecks, moving or not, was addressed in \cite{andreaianov2010NM, ColG, GNPT} and the case of multiple moving bottlenecks was already considered in \cite{LMP10}.

Finally, the terminology ``moving bottleneck'' is known also in transport engineering literature: see, among others, \cite{GH92, LaDa05, LLG, New98}. In these papers the problem of moving bottlenecks, both with a given path \cite{GH92, LLG, New98} and with a mutual interaction with the surrounding flow \cite{LaDa05}, is studied, but without a detailed description of the resulting schemes and of their convergence.

\paragraph{Goal.} We propose two different numerical algorithms, the first allowing theoretical investigations, while the second allowing an actual implementation on a computer.

The first algorithm, called WFT-ODE, combines the Wave Front Tracking (WFT) method \cite{B,HRbook} to solve (\ref{lwr})
with an exact method to solve (\ref{ode}). The moving bottleneck position
is traced taking into consideration interactions with (shock or rarefaction) waves.
Focusing on bounded variation data, we establish a convergence result
for the approximate solution towards a solution to (\ref{lwr})--(\ref{ode}).
We are also able to provide a convergence rate for $y$  in terms of total
variation of initial datum of density, namely $TV(\bar\rho)$. To achieve the estimates, we measure the distance
among successive approximations by means of  tangent vectors. The latter represents distances among
discontinuities (shocks or rarefaction shocks) for car density, and  the difference of positions for moving bottleneck.

The second algorithm, called GOF, is a fractional step method combining the classical Godunov method
\cite{GO} for (\ref{lwr}) with an exact method for (\ref{ode}). There are two reasons for using the Godunov method:
the easy implementation 
and the importance of this method in the transportation literature \cite{Leb,LLG}.

The GOF method is then extended to the case of multiple bottlenecks. For that, we consider two models. The first allows slow vehicles to pass each other, while the second inhibits such a possibility.
The latter is more challenging from theoretical and numerical point of view, thus we focus on this case. Both methods are useful to deal, for instance, with the case of many buses on the same urban route \cite{GLM11}.

\paragraph{Paper organization.}
In Section \ref{sec:basic} we provide basic definitions and preliminary results.
The theoretical scheme WFT-ODE is introduced in Section \ref{sec:WFT-ODE},
while convergence results are given in Section \ref{sec:WFT-ODEconvergence}.
The second scheme GOF is described in Section \ref{sec:GOF}.
Then we deal with multiple bottlenecks in Section \ref{sec:multi}.
Finally, numerical simulations are reported in Section \ref{sec:simulations}.

\section{Basic definitions and preliminary results}\label{sec:basic}
In this section we provide the definition of solution to the system (\ref{lwr})--(\ref{ode})
and we detail the choice of the flux $f$ and the velocity $w$.
We also give a description of Wave Front Tracking (WFT) algorithm for the
case of a flux function depending on the  space variable, as it occurs for (\ref{lwr}).
Finally, a result on possible interactions of the moving bottleneck
with fronts in the WFT solution is proved. The latter will be extensively used in the rest of the paper.

We denote by $\mathcal{BV}(\R)$ the space of functions with bounded variation.
For every $T>0$, the definition of solution to the system (\ref{lwr})--(\ref{ode}) reads as follows.

\begin{definition}\label{def:sol}
A vector $(\rho,y)$ is a solution to (\ref{lwr})--(\ref{ode}) in
$[0,T]$ if
$\rho(t,\cdot)\in\mathcal{BV}(\R)$ for
a.e.\ $t\in [0,T]$
which solves the PDE in the sense of distributions, that is
\begin{equation*}
\int_{0}^{T}\int_{\R}
\Big(\rho(t,x)\partial_t\phi(t,x)+f(\rho,x-y(t))\partial_x\phi(t,x)\Big)dxdt +
\int_{\R}\bar\rho\phi(0,x)dx=0,
\end{equation*}
for any differentiable function $\phi(t,x)$ with compact support in
$t\geq
0$.
Moreover, the position of the moving bottleneck $y(t)$ solves the ODE in $[0,T]$
in the sense of Filippov, namely $y(t)$ is an absolutely continuous
function such that $y(0) = y_0$ and
\begin{equation*}
\dot y\in F(\rho)=\cco\{w(\rho):\rho\in\mathcal{I}[\rho(t,y(t)-),\rho(t,y(t)+)]\}
\end{equation*}
for a.e.\ $t\in [0,T]$, where the set $\mathcal{I}[a,b]$ is defined
as the smallest interval containing $a$ and $b$.
\end{definition}

We assume hereafter that the maximal density of cars is $\rho_{max}$=1. We denote by $w(\rho)$ the velocity of the moving bottleneck and we assume that
\bee \label{vel_bus}
w(\rho)=w_{max}(1-\rho)
\ede
for some constant $w_{max}>0$. We denote by $v(\rho,x-y)$ the velocity of cars and we assume that
\bee \label{flusso}
f(\rho,x-y)= \rho v(\rho,x-y) ~\textrm{ and }~ v(\rho,x-y)=\varphi(x-y)(1-\rho)\,,
\ede
where $\varphi\in C^1(\R;(0,+\infty))$ and it is strictly decreasing in $(-\beta,0]$ and strictly increasing in $[0,\beta)$ for some $\beta>0$.
We also assume that there exist two positive constants
$\vmin$, $\vmax$ such that
\bee \label{eq:vmax}
0<\vmin\leq \varphi(\zeta)\leq \vmax\quad \textrm{for\ any\ }\zeta\in\R
\ede
and that $\vmin$ is reached for $\zeta=0$ and $\vmax$ is reached for any $\zeta\in\R\backslash[-\beta,\beta]$, see Figure\ \ref{fig:phi}(a).

\begin{figure}
\begin{center}
\begin{tabular}{cc}
\includegraphics[height=5cm]{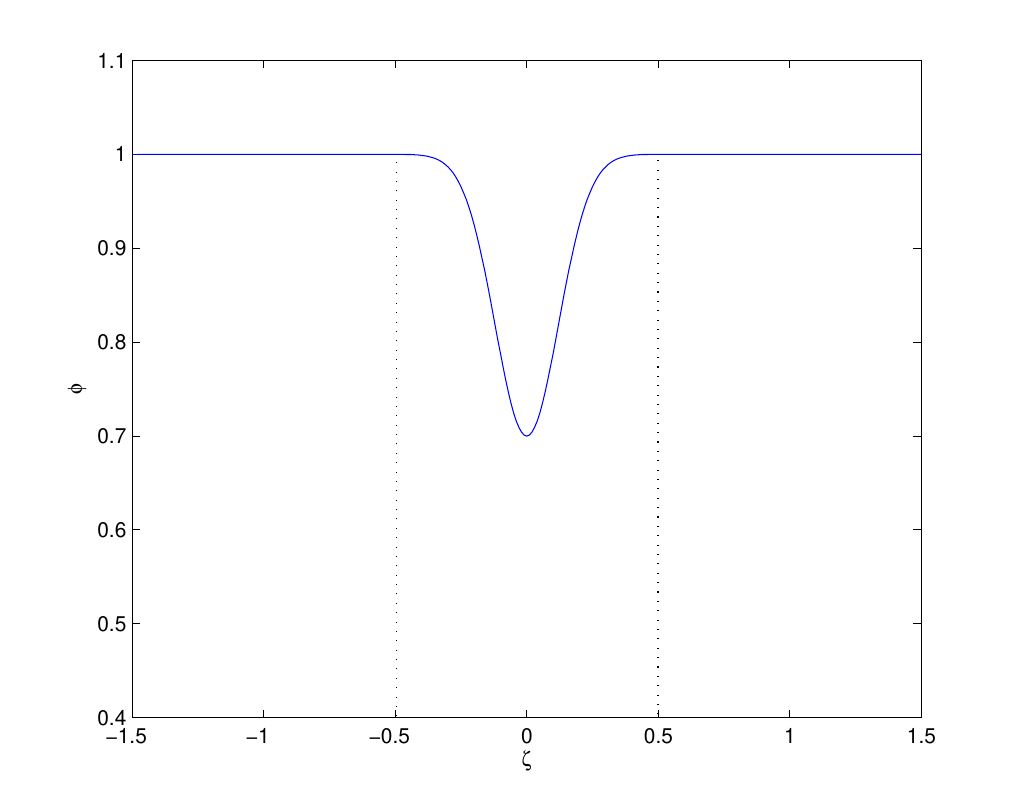} &
\includegraphics[height=5cm]{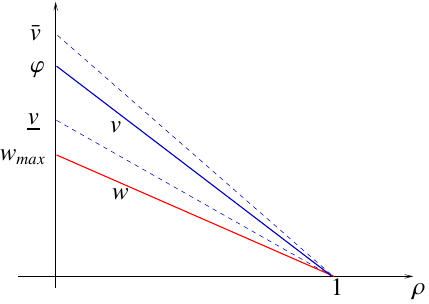} 
\\
(a) & (b)
\end{tabular}
\end{center}
\caption{(a) A possible choice for $\varphi(\zeta)$ with $\beta=0.5$, $\underbar v =0.7$, $\vmax =1$. (b) Cars' and moving bottleneck's velocity as a function of $\rho$ (for a fixed $\zeta=x-y$).}
\label{fig:phi}
\end{figure}

Note that, in particular $\varphi'(\zeta)$ is compactly supported and, for any fixed density $\rho$, the minimal velocity of cars $\vmin$ is taken when they overtake the moving bottleneck, i.e.\ for $x=y$.
We finally assume
\begin{equation}\label{vmaxmin>vbmax}
\vmin>w_{max}.
\end{equation}
The last assumption is crucial for the model: it means that \textit{cars can overtake the bottleneck}, see Figure\ \ref{fig:phi}(b).

\subsection{Preliminaries for Wave-Front Tracking algorithms}\label{WFT}

As it is well known, a solution to the Cauchy problem of a conservation law can be constructed using the WFT algorithm.
Roughly speaking, first a $\mathcal{BV}$ initial datum is approximated
by a piecewise constant function via sampling (thus with a smaller $\mathcal{BV}$ norm).
For small times, a piecewise constant weak solution is obtained by piecing together
solutions to Riemann problems, where rarefactions are replaced by a fan of rarefaction shocks.
Then, when waves interact, a new Riemann problem is generated and solved
again approximating rarefactions by rarefaction shocks' fans and so on.
For details we refer the reader to
\cite{B} for the general theory and to \cite{GPbook} for the case of networks.

In our case the flux does not depend only on $\rho$, but also on $x$ and $t$ (via $y(t)$), so the WFT algorithm needs to be modified.
Moreover, we will evolve, at the same time, the position of the bottleneck $y(t)$
according to equation (\ref{ode}).
In order to perform the construction we first need a preliminary result
comparing the speed of the bottleneck with that of neighboring waves of the
WFT approximation of $\rho$.
\begin{lemma}\label{lemma3p1}
Assume (\ref{vel_bus})--(\ref{vmaxmin>vbmax}) hold true.
Assume that the discontinuity $(\rho_l,\rho_r)$, located at $x$, is the
closest to the bottleneck position $y$ and there exists
\begin{equation}\label{eq:cond:keylemma}
\rho_{min}>\frac{\vmax-w_{max}}{2 \vmax-w_{max}}
\end{equation}
such that $\rho_l,\rho_r \geq\rho_{min}$.
Define
\begin{equation}\label{eq:def-mu}
\mu= \frac{(2 {\underbar{v}} - w_{max})}{2}
\left( \rho_{min}-\frac{\vmax-w_{max}}{2 \vmax-w_{max}}\right),
\end{equation}
then
\begin{equation}\label{tesiLemma3p1}
w(\rm)>\lambda+\mu \quad \textrm{ and } \quad w(\rp)>\lambda+\mu.
\end{equation}
\end{lemma}
We remark that
$$
\frac{\vmax-w_{max}}{2 \vmax-w_{max}} \in \left (0,\frac 1 2
 \right ).
$$
\begin{proof}
Recall that the discontinuity travels with speed:
	\bee\label{rh}
	\lambda(\rp,\rm,\zeta)=\frac{\rp v(\rp,\zeta)-\rm v(\rm,\zeta)}{\rp-\rm}.
	\ede
We divide the proof in two cases:
\begin{enumerate}
\item The wave is a shock wave.
\item The wave is a rarefaction-shock wave.
\end{enumerate}
1.
It must be $\rp>\rm$. Using (\ref{rh}), to prove inequalities in (\ref{tesiLemma3p1}) we consider the following:
\begin{equation}\label{lemma3p1_casoshock}
w_{max}(1-\rm)>\frac{\varphi(\zeta)\rp(1-\rp)-\varphi(\zeta)\rm(1-\rm)}{\rp-\rm}+\mu\,,
\end{equation}
and
\begin{equation}\label{lemma3p1_casoshock_bis}
w_{max}(1-\rp)>\frac{\varphi(\zeta)\rp(1-\rp)-\varphi(\zeta)\rm(1-\rm)}{\rp-\rm}+\mu.
\end{equation}
Let us first prove (\ref{lemma3p1_casoshock_bis}). We notice that
\begin{align*}
\varphi(\zeta)\Big(\rp(1-\rp)-\rm(1-\rm)\Big)=&\varphi(\zeta)\Big(\rp-\rp^2-\rm+\rm^2\Big)\\
= &\varphi(\zeta)(\rp-\rm)(1-(\rm+\rp)).
\end{align*}
Thus, \eqref{lemma3p1_casoshock_bis} can be written as
	\begin{displaymath}
	w_{max}(1-\rp) > \varphi(\zeta)(1-(\rm+\rp))+\mu
	\end{displaymath}
and then as
	\begin{displaymath}
	\varphi(\zeta)\rm>(1-\rp)(\varphi(\zeta)-w_{max})+\mu.
	\end{displaymath}
Since from \eqref{eq:vmax} and \eqref{vmaxmin>vbmax} we have $\varphi(\zeta)\geq\vmin>w_{max}$, the worst case is when $\rm=\rho_{min}$ and  $\rp=\rho_{min}+\varepsilon$ ($\varepsilon>0$ arbitrarily small):
	\begin{displaymath}
	\varphi(\zeta)\rho_{min}>(1-(\rho_{min}+\varepsilon))(\varphi(\zeta)-w_{max})+\mu;
	\end{displaymath}
then it is sufficient to prove
	\begin{displaymath}
	\rho_{min}> (1-\varepsilon) \frac{\varphi(\zeta)-w_{max}}{2 \varphi(\zeta)-w_{max}}+\frac{\mu}{2\vmin-w_{max}}.
	\end{displaymath}
If we define $\psi(z)=\frac{z-w_{max}}{2 z-w_{max}}$, with $z=\varphi(\zeta) \in [\vmin,\vmax]$, we find that $\psi'(z) >0$, and consequently:
	\bee\label{psi}
	\max_{z \in [\scriptsize \vmin \normalsize ,\vmax]}{\psi(z)}= \psi(\vmax)=\frac{\vmax-w_{max}}{2 \vmax-w_{max}}.
	\ede
Therefore, in the worst case (taking also $\varepsilon=0$)
	\begin{displaymath}
	\rho_{min}>  \frac{\vmax-w_{max}}{2 \vmax-w_{max}}+\frac{\mu}{2\vmin-w_{max}} ,
	\end{displaymath}
which is true for $\mu$ defined as in (\ref{eq:def-mu}). Since $\rp>\rm$, (\ref{lemma3p1_casoshock_bis}) implies (\ref{lemma3p1_casoshock}) and then we conclude the proof.
\vskip0.2cm
\noindent 2. In the case of rarefaction-shock we have $\rho_r = \rho_\ell - \delta_\nu$, where $0<\delta_\nu<\frac1{\nu}$ is a parameter that will be defined
by the WFT algorithm.
We only prove the inequality in (\ref{lemma3p1_casoshock}), then (\ref{lemma3p1_casoshock_bis}) follows because $\rm>\rp$. The inequality in \eqref{lemma3p1_casoshock} can be rewritten as
	\begin{displaymath}
	w_{max}(1-\rp)>\frac{\varphi(\zeta)\rp(1-\rp)-\varphi(\zeta)\rm(1-\rm)}{\rp-\rm}
	 + w_{max}\delta_\nu+\mu.
	\end{displaymath}
Reasoning as above, we have
	\begin{displaymath}
	\varphi(\zeta)\rm>(1-\rp)(\varphi(\zeta)-w_{max})+ w_{max}\delta_\nu+\mu.
	\end{displaymath}
Now the worst case occurs when $\rp=\rho_{min}$, $\rm=\rho_{min}+\delta_\nu$:
	\begin{displaymath}
	\varphi(\zeta)(\rho_{min}+\delta_\nu)>(1-\rho_{min})(\varphi(\zeta)-w_{max})
	+ w_{max}\delta_\nu+\mu,
	\end{displaymath}
i.e.
	\begin{align}\label{rhomin2}
	\rho_{min}> &\frac{\varphi(\zeta)-w_{max}}{2 \varphi(\zeta)-w_{max}} +
	\frac{\delta_\nu(w_{max} - \varphi(\zeta))}{2 \varphi(\zeta)-w_{max}}
	+\frac{\mu}{2\vmin-w_{max}}\nonumber\\
	&=\frac{(1-\delta_\nu)(\varphi(\zeta)-w_{max})}{2 \varphi(\zeta)-w_{max}}
	+\frac{\mu}{2\vmin-w_{max}}.
	\end{align}
Then we conclude as in step 1.
\end{proof}

The WFT algorithm will need to be further split because of source effect
due to the $x$ dependence of the flux. Indeed, consider a general equation of the type
\begin{equation}\label{eq}
\partial_t\rho(t,x)+\partial_x \Big[g(\rho(t,x),x) \Big]=0,
\end{equation}
where we assume
$$
\textrm{\textbf{(H1)}}\qquad
\begin{array}{l}
g \textrm{ is Lipschitz continuous in both variables}; \\
\textrm{for every } \rho, ~g(\rho,\cdot)\in C^2(\R);  \\
\max\limits_{\rho\in[0,1]} \int_\R|\partial_{xx}g(\rho,x)|dx<+\infty; \\
\textrm{for every } x, ~g(\cdot,x) \textrm{ is concave}.
\end{array}
$$
Therefore, formally deriving in space the flux function $g$ depending both on the density of cars $\rho(t, x)$ and on the position of cars $x$, one gets:
\begin{equation}\label{lc_con_g}
\partial_x [g(\rho,x)]=\partial_\rho g ~ \partial_x \rho + \partial_x g\,,
\end{equation}
so that equation (\ref{eq}) can be written as a conservation law with source term
\begin{equation}\label{lc+st}
\partial_t \rho + \partial_\rho g ~\partial_x \rho = - \partial_x g.
\end{equation}
In our case we have
$$
g(\rho,x;y)=\rho\varphi(x-y)(1-\rho)
$$
and the source term reads
$$
-\partial_x g(\rho,x;y)=-\rho(1-\rho)\partial_x\varphi(x-y).
$$
(see Figure\ \ref{fig:mgx}).
\begin{figure}[h!]
    \begin{center}
    \includegraphics[height=5cm]{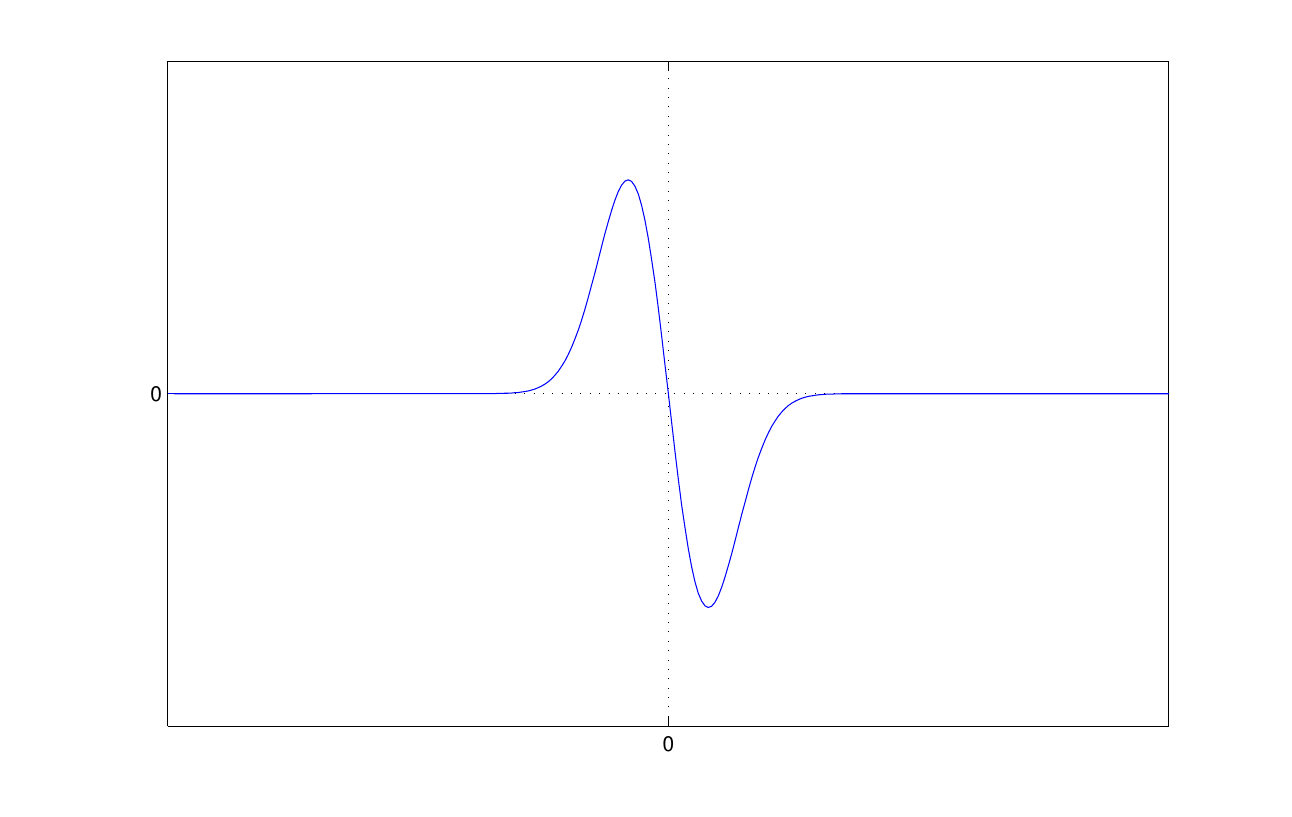}
	\caption{Shape of the  term $-\partial_x\varphi(x-y)$ as a function of $\zeta=x-y$.}
	\label{fig:mgx}
    \end{center}
\end{figure}
Note that the assumptions \textbf{(H1)} are satisfied and that the $x$ dependence in the flux is indeed effective only  for $\zeta\in(-\beta, \beta)$, that is when cars are close to the bottleneck's position.

Equation (\ref{eq}) can be solved applying an operator splitting method in which one alternates between solving the homogeneous conservation law
\begin{equation}\label{wft_pdeomogenea}
\partial_t\rho+\partial_{\rho} g ~\partial_x\rho =0
\end{equation}
and the ordinary differential equation
\begin{equation}\label{wft_ode}
\partial_t\rho = -\partial_x g.
\end{equation}

\section{The semi-discrete WFT-ODE scheme and its properties}\label{sec:WFT-ODE}
We are now ready to define the WFT-ODE scheme.
In this section we describe a semi-discrete scheme for the coupled dynamics (\ref{lwr})--(\ref{ode}).
The idea is to combine the WFT method for the conservation laws with exact solution
to the ODE of the moving bottleneck. Indeed, since the approximate solution given
by WFT is piecewise constant (in time and space), we can get a simple ODE with piecewise constant (in time) right-hand side
for the moving bottleneck.\\
\indent Let us now describe in detail the WFT-ODE scheme.
Fix $T>0$, then
for every $\nu\in\N$ and $\Delta t>0$,
we define an approximate solution $(\rho_\nu,y_\nu)$ on the time interval $[0,T]$
via the following steps.
 
\begin{itemize}
\item[Step 1.]
Let $\bar\rho\in\mathcal{BV}(\R)$ the initial datum.
We define a piecewise constant function $\bar\rho_\nu$ with $N$ discontinuities $s_1<\cdots<s_N$, by approximating $\bar\rho$ as follows (Figure 3):
$$
\rho_\nu(x)=\bar\rho(j2^{-\nu}), \quad x\in[j2^{-\nu},(j+1)2^{-\nu}]\cap[-\nu,\nu], \quad j\in\mathbb Z, \quad \nu\in\mathbb N.
$$

\begin{figure}[h!]
\begin{center}
\begin{overpic}
[width=0.65\textwidth]{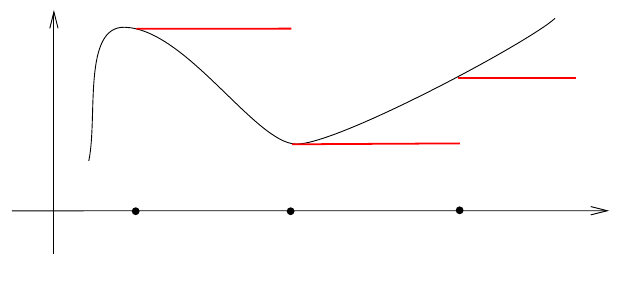}
\put(14,10){$(j-1)2^{-\nu}$} \put(43,10){$j2^{-\nu}$}  \put(68,10){$(j+1)2^{-\nu}$}
\put(40,45){$\bar\rho_\nu$} \put(35,35){$\bar\rho$} \put(98,10){$x$}
\end{overpic}
\end{center}
\caption{Sampling $\bar\rho_\nu$ for WFT.}
\label{fig:samplenu}
\end{figure}
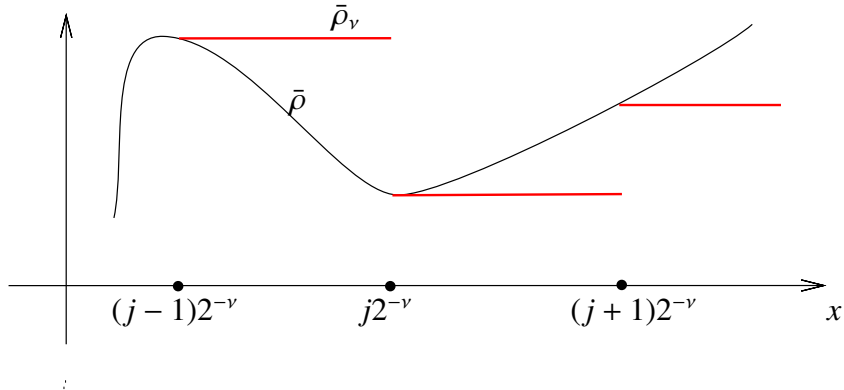

Notice that $TV(\bar \rho_\nu)\le TV(\bar \rho)$.

\item[Step 2.]
At time $t=0$, we solve all Riemann problems at each discontinuity point of $\bar\rho_{\nu}$ for the homogeneous equation (\ref{wft_pdeomogenea}) and
approximate every rarefaction wave with a rarefaction fan, formed by
rarefaction shocks (non-entropic shocks) of strength less than $\delta_\nu$ travelling with speed given by the Rankine-Hugoniot condition.\\
More precisely: Assume, for simplicity, that $y_0\in ]s_j,s_{j+1}[$ for some $j$
(i.e.\ $y_0\not= s_i$ for every $i=1,\ldots,N$.)
For small times we can get an approximate solution by solving the equations:
\[
\dot s_i=\varphi(s_i(t)-y(t))
\frac{\rho(t,s_i(t)-)(1-\rho(t,s_i(t)-))-
\rho(t,s_i(t)+)(1-\rho(t,s_i(t)+))}
{\rho(t,s_i(t)-)-\rho(t,s_i(t)+)}
\]
for the discontinuities $s_i$ of $\rho_\nu$ and the equation:
\[
\dot y(t)=w(\rho(t,s_j+)),
\]
for the bottleneck, which gives $y(t)=y_0+tw(\rho(t,s_j+))$.\\
If two wave fronts $s_i$ and $s_{i+1}$ interacts, then we solve the
new Riemann problem. Because of Lemma \ref{lemma3p1},
$s_j(t)<y(t)$ for every $t\in[0,\Delta t]$.
If $y(\bar t)=s_{j+1}(\bar t)$ for $\bar t\in[0,\Delta t[$, then the bottleneck interacts
at $\bar t$ with the wave $s_{j+1}$ and we simply change the equation
for the bottleneck to:
\[
\dot y(t)=w(\rho(t,s_{j+1}+)),
\]
and so on for all the possible interactions with $\rho_\nu$ waves up to time
$\Delta t$.\\
Finally, possibly slightly modifying the wave speeds, we can assume that $y_\nu(\Delta t) \neq s_i(\Delta t)$ for any $j$.

\item[Step 3.]
To take into account the source term $\rho_\nu(\Delta t,x)$ is updated in order to include the effect of the source term.
This is done by means of one step of the explicit Euler method approximating the ODE (\ref{wft_ode}) by:
\begin{equation}\label{updatedsourceterm}
\rho_\nu^{source}(\Delta t,x)=\rho_\nu(\Delta t,x)+\Delta t (-\partial_x g(\rho_\nu(\Delta t,x),x))\quad \textrm{for any } x\in\R .
\end{equation}
After this modification the function $\rho_\nu^{source}$ is no longer piecewise constant and we need to apply a specific sampling to ensure accurate estimates of continuous dependence.

\item[Step 4.]
Let $\bar y=y(\Delta t)$.
We update $\rho_\nu^{source}$ by sampling only on the interval $[\bar y-\beta,\bar y+\beta]$ (where $\rho^{source}$ is not piecewise constant.)
Let $x_i$, $i=1,\ldots,M$,  be the point of discontinuity of $\rho_\nu^{source}$ on
$[\bar y-\beta,\bar y+\beta]$.
First we insert the points $y_j=\bar y-\beta+j\,\frac{2\beta}{2^\nu}$,
$j=1,\ldots , 2^\nu$. Moreover, if there exists $i$ such that
no $y_j$ is in the interval $]x_i,x_{i+1}[$, then we insert
the additional point $y_{i}=\frac{x_i+x_{i+1}}{2}$.
Again possibly slightly modifying the wave speeds we can assume that
$x_i\not= y_j$ for every $i$ and $j$.
We define $\rho^{sampled}$ by
sampling $\rho^{source}$ on subintervals generated by the partition:
$P=\{x_i\}\cup\{y_j\}\cup \{y_i\}$ with the following rules:\\
- On intervals $[x_i,y_j]$ we define
$\rho_\nu^{updated}=\rho_\nu^{source}(x_i+)$.\\
- On intervals $[y_j,x_{i+1}]$ we define
$\rho_\nu^{updated}=\rho_\nu^{source}(x_{i+1}-)$.\\
- On all other intervals $[a,b]$
we define $\rho_\nu^{updated}=\rho_\nu^{source}(a+)$.\\
The rationale behind this definition is the following.
To obtain neat estimates of increase in jumps size of shocks and rarefaction shocks
we need always to sample  to the left and right of the discontinuities
which were generated by the WFT step. In the sequel, we shall refer to these new generated (at positive time) waves  as $s_{b}$ waves, being generated by the source term due to the bottleneck.
\\

We can restart the procedure at step 2 in the interval $[\Delta t, 2\Delta t]$ and so on.
\end{itemize}

It is important to note that, while the sampling procedure does not increase the total variation, the updating step (\ref{updatedsourceterm}) can do it. For the first time step,
the increment can be estimated easily as follows. Let us denote by $\omega(x;\rho_\nu,\Delta t)$ the contribution given by the source term at time $\Delta t$, i.e.\
$$\omega(x;\rho_\nu,\Delta t):=-\Delta t \partial_x g(\rho_\nu(\Delta t,x),x).$$
Recalling the assumption \textbf{(H1)}, we have
\begin{eqnarray}
TV(\rho_\nu^{updated})
&=&\sup_{\left\{\textrm{partitions }\{x_j\}_j \textrm{ of } \R\right\}}\sum_{j}|\rho_\nu(\Delta t,x_j)-\rho_\nu(\Delta t,x_{j-1})+\omega(x_j;\rho_\nu,\Delta t)-\omega(x_{j-1};\rho_\nu,\Delta t)|
 \nonumber \\ \nonumber
&\leq &\sup_{\left\{\textrm{partitions }\{x_j\}_j \textrm{ of } \R\right\}}\sum_{j}|\rho_\nu(\Delta t,x_j)-\rho_\nu(\Delta t,x_{j-1})|+\int_\R|\omega'(x;\rho_\nu,\Delta t)|dx.
\end{eqnarray}
Then, at every time step the additional total variation is bounded by $\max\limits_{\rho}\|\omega'(\cdot ;\rho,\Delta t)\|_{L^1}$,
and at final time $T$, namely after $\frac{T}{\Delta t}$ time steps, the total increment is bounded by  $T \max\limits_{\rho}\|\partial_{xx}g(\rho, \cdot) \|_{L^1}$. More precisely we have the following:
\begin{lemma}\label{le:step3-tv-est}
The increase of total variation in $\rho_\nu$ at Step 3 is bounded by
$K_\varphi\ \Delta t$ where
\begin{equation}\label{additionalTVduetosourceterm}
K_\varphi : =\max_{\rho\in[0,1]}\left\{ \rho(1-\rho)\int_\R|\partial_{\zeta\zeta}\varphi(\zeta)|d\zeta\right\}
=
\frac{1}{4}  \int_\R|\partial_{\zeta\zeta}\varphi(\zeta)|d\zeta.
\end{equation}
\end{lemma}

Bounds on the $\mathcal{BV}$ norm and on number of waves and interactions are directly obtained as in the scalar case,
when the flux is not depending on $x$, provided this dependence does not create resonance effects among waves. Indeed, we can rewrite a scalar conservation law with a $x$-dependent flux $g(\rho,x)$ as a $2\times2$ system as follows:
\begin{equation*}
\begin{cases}
\rho_t + g(\rho,z)_x = 0 & \\
z_t = 0. &
\end{cases}
\end{equation*}
The above system is strictly hyperbolic, and therefore standard WFT analysis may be carried out, provided
$\partial_\rho g(\rho, z)\neq 0$. In our case, $g(\rho, x;y(t)) = \rho(1-\rho)\varphi(x-y(t))$ and the above condition corresponds to the requirement that the bottleneck's position $y(t)$, where the $x$-dependence is effective as noticed before, travels with a speed uniformly different with the one of the closest $\rho$ waves.
This is precisely the statement of Lemma \ref{lemma3p1}.

The construction thus relies on Lemma \ref{lemma3p1}, which in turn is applicable
if $\rho_\nu>\frac{\vmax-w_{max}}{2 \vmax-w_{max}}$, see (\ref{eq:cond:keylemma}).
Due to the source term $-\partial_x g$ this is not guaranteed for all times
even if the initial data satisfies (\ref{eq:cond:keylemma}). However
we have the following:
\begin{lemma}\label{le:T-est}
Consider an initial datum $\bar{\rho}\in\mathcal{BV}(\R)$ and assume
$\eta=\min_x \frac{1}{\bar{\rho}(x)}\ \frac{\vmax-w_{max}}{2 \vmax-w_{max}}<1$.
Then  for all positive times $T$ such that:
\begin{equation}\label{eq:cond:T}
T< \frac{- \ln (\eta)}{\|\varphi'\|_{\infty}},
\end{equation}
the condition (\ref{eq:cond:keylemma}) of Lemma \ref{lemma3p1} is verified
for $\Delta t$ sufficiently small.
\end{lemma}

Notice that for $\bar{v}\to w_{max}$ we have $\eta\to 0$, thus $T\to \infty$.
In other words, if the maximal speed of cars is close the that of the moving bottleneck
then the WFT-ODE scheme is defined for large times.
\begin{proof}
Notice that solutions to the ODE linked to the source term (\ref{wft_ode}) satisfy
$\partial_t \rho=\rho(1-\rho)\partial_x\varphi\geq -\|\varphi'\|_{\infty} \rho$,
thus the solution satisfies $\rho(t)\geq e^{-t\|\varphi'\|_{\infty}} \rho(0)$.
This gives the estimate (\ref{eq:cond:T}). Since Step 3 is an approximation
of the source term, the estimate remains valid for $\Delta t$ sufficiently small.
\end{proof}

\section{Convergence of the WFT-ODE scheme}\label{sec:WFT-ODEconvergence}
In this section we prove the convergence of the WFT-ODE scheme.
The BV estimates on the WFT scheme allows to pass to the limit by standard
arguments. On the other side, the convergence of the the bottleneck position
requires additional work, in the same spirit of \cite{BP08}. However,
our case is more complex due to the complete coupling of the PDE for car density
and the ODE for the bottleneck position.
As in \cite{BP08}, to achieve this goal we use the technique of generalized tangent vectors \cite{B2, B3},
to the WFT approximate solutions $\rho_\nu$ and to the bottleneck position $y_\nu$.

More precisely, after introducing the tangent vector technique,
we consider $\rho_{\nu+1}$ as obtained from $\rho_\nu$ by shifts
of discontinuities, which generate generalized tangent vectors.
Then we need to estimate:\\
1) The increase in the norm of tangent vector to $\rho_\nu$ due to the WFT algorithm.\\
2) The increase in the norm of  tangent vector to $\rho_\nu$ due to the source term.\\
3) The increase in the norm of tangent vector to $y_\nu$ due to interactions with waves of $\rho_\nu$.

\subsection{Generalized tangent vectors}
The technique is based  on the idea of considering $L^1$ as
a Finsler manifold with generalized tangent vectors with appropriate norm.
We first  consider the subspace of piecewise
constant\index{piecewise constant} functions and ``generalized tangent vectors''
 consisting of two components $((v^\theta,\xi^\theta ),\eta^\theta)$, where $\xi^\theta$ describes the {infinitesimal}
 displacement of discontinuities and the scalar $\eta^\theta$ is the {infinitesimal} shift of the car trajectory. Let us take a family of piecewise constant functions $\{\rho ^{\theta }\}_{\theta \in [0,1]}$, each of
which has the same number of jumps, say at the points
$s_{1}^{\theta}<...<s_{N}^{\theta }$. Let us define the function
\begin{equation*}
v^{\theta }(x)\dot{=}\lim_{h\rightarrow 0}\frac{\rho
^{\theta +h}(x)-\rho ^{\theta }(x)}{h},
\end{equation*}
and also the quantities
\begin{equation*}
\xi _{k }^{\theta }\dot{=}\lim_{h\rightarrow 0}\frac{s_{k
}^{\theta +h}-s_{k }^{\theta }}{h},\qquad k =1,...,N,
\end{equation*}
and
\begin{equation*}
\eta ^{\theta }\dot{=}\lim_{h\rightarrow 0}\frac{y^{\theta +h}-y^{\theta }}{h}.
\end{equation*}

Note that $v^\theta$ is not defined if $x=s^\theta_k$, $k=1,\ldots,N$. Indeed, the contribution of the jumps to the tangent vector is given by $\{\xi^\theta_k\}_k$, which take into account the presence of the discontinuities.
Moreover, $v^\theta$ solve the linearized equation along the trajectory,
see \cite{B}
while the shifts change only at interactions times or, as detailed in next remark,
due to Step 3 in the WFT-ODE scheme.

\begin{remark}\label{rem:shftsb}
It is worth observing that, as a consequence of the shift $\eta ^{\theta}$ of the moving bottleneck's position, the waves $s_b$ generated in Step 4 of the scheme
will be all shifted by the same quantity.
\end{remark}

Then we say that the path $\gamma:\theta
\rightarrow (\rho ^{\theta },y^\theta)$ admits tangent vectors
$((v^{\theta},\xi ^{\theta }),\eta^{\theta })\in T_{\rho^\theta}\dot{=}
L^{1}(\mathbb{R},\mathbb{R}) \times \mathbb{R}^{N}\times \mathbb{R}$.
Note that in general such path is not differentiable w.r.t. the usual differential
\index{differential structure}
structure of $L^{1}$. 
One can compute the $L^{1}$-length of the path $\gamma$ in the following way:
\begin{equation}\label{eq:ch2:length}
\left\Vert \gamma \right\Vert
_{L^{1}}=\int\limits_{0}^{1}\left\Vert v^{\theta }\right\Vert
_{L^{1}}d\theta +\sum\limits_{k
=1}^{N}\int\limits_{0}^{1}\left\vert \rho^{\theta }(s_{k
}+)-\rho^{\theta }(s_{k }-)\right\vert \left\vert \xi _{k
}^{\theta }\right\vert d\theta +\int\limits_{0}^{1}|\eta^{\theta }|d\theta.
\end{equation}

According to (\ref{eq:ch2:length}), in order to compute the
$L^{1}$-length of a path $\gamma $, we integrate the norm of its
tangent vector which is defined as follows:
\begin{equation*}
\left\Vert ((v^\theta,\xi^\theta ),\eta^\theta)\right\Vert \dot{=}\left\Vert v^\theta\right\Vert
_{L^{1}}+\sum\limits_{k =1}^{N}\left\vert \Delta \rho^\theta_{k
}\right\vert \left\vert \xi^\theta_{k }\right\vert + |\eta^\theta|,
\end{equation*}
where
$\Delta \rho^\theta_{k }$ is the jump across the discontinuity $s^\theta_{k }$.
Notice that this is not the usual $L^1$ length of a path
since there is the additional component $\eta$.
\smallskip

\noindent Let us introduce the following definition.
\begin{definition}
We say that a continuous map
$\gamma :\theta \rightarrow (\rho^{\theta },y^\theta)\dot{= }\gamma (\theta )$
from $[0,1]$ into $L_{loc}^{1}\times\R$ is a regular path\index{regular path}
if the following holds.
All functions $\rho ^{\theta }$ are piecewise constant, with the same
number of jumps, say at $ s_{1}^{\theta }<...<s_{N}^{\theta }$ and
coincide outside some fixed interval $\left] -M,M\right[ $.
Moreover, $\gamma$ admits a generalized tangent vector
\index{generalized tangent vector}
$D\gamma (\theta )= ((v^{\theta },\xi ^{\theta }),\eta^\theta)\in T_{\gamma
(\theta )}=L^{1}(\mathbb{ R};\mathbb{R})\times\mathbb{R}^{N}\times \mathbb{R}$,
continuously depending on $\theta$.
\end{definition}

\subsection{WFT approximations with shift evolution}
Now, given two successive approximations of the initial data $\bar\rho$, namely $\bar\rho_\nu$ and $\bar\rho_{\nu+1}$, obtained by sampling, respectively, at lattice points of mesh $2^{-\nu}$ and $2^{-(\nu+1)}$, we can connect them by a regular path as follows (Figure 4):
\begin{eqnarray*}
& &\rho^\theta(x,t=0)= \\ \nonumber
& & =\left\{
 \begin{array}{lll}
 \bar\rho(j2^{-\nu}), \  x \in [j2^{-\nu}, j2^{-\nu} + 2^{-(\nu+1)}(1+\theta)),\\
  \bar\rho(j2^{-\nu} + 2^{-(\nu+1)}), \   x \in [j2^{-\nu} + 2^{-(\nu+1)}(1+\theta), (j+1)2^{-\nu}).
  \end{array}
   \right.
\end{eqnarray*}

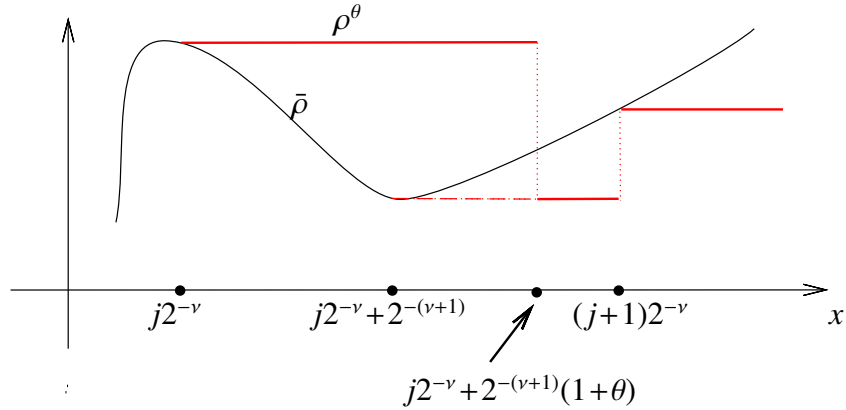
\begin{figure}[h!]
\begin{center}
\begin{overpic}
[width=0.65\textwidth]{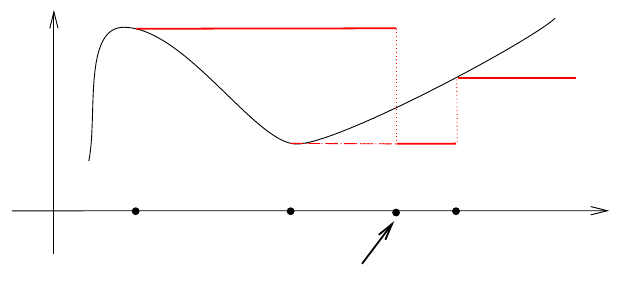}
\put(18,10){$j2^{-\nu}$} \put(68,10){$(j\!+\!1)2^{-\nu}$}
\put(37,10){$j2^{-\nu}\!+\!2^{-(\nu+1)}$} \put(48,1){$j2^{-\nu}\!+\!2^{-(\nu+1)}(1\!+\!\theta)$}
\put(40,45){$\rho^\theta$} \put(35,35){$\bar\rho$} \put(98,10){$x$}
\end{overpic}
\end{center}
\caption{The regular path $\rho^\theta(x,t=0)$, with $\theta\in(0,1)$.}
\label{fig:rho0nu}
\end{figure}

Notice that such path satisfies $\rho^0 = \bar\rho_{\nu+1}$ and $\rho^1 = \bar\rho_{\nu}$. Moreover, $\rho^\theta$ is differentiable everywhere
(possibly) except for $\theta = 0, 1$, where the number of jumps changes. Finally we have:
\begin{lemma}\label{lem:initialshifts}
The initial data $\bar{\rho}_\nu$ and $\bar{\rho}_{\nu+1}$ can be connected by a regular path
with shifts whose norm is bounded by $2^{-(\nu+1)}$.
\end{lemma}

Since $\rho^\theta$ is piecewise constant, $v^\theta(x)$ is equal to 0 (where it is defined), then $\|v^\theta\|_{L^1}=0$. At time $t=0$, we also have $\eta^\theta=0$ since $y_\nu(0)=y_{\nu+1}(0)$. By (\ref{eq:ch2:length}), we have
$$
\|\gamma\|_{L^1}=\sum\limits_{k
=1}^{N}\int\limits_{0}^{1}\left\vert\Delta\rho^\theta_k\right\vert \left\vert \xi _{k
}^{\theta }\right\vert d\theta.
$$

Now we consider the evolution in time of $\gamma$, denoted by $\gamma_t$, in the sense that we perform the wave front tracking for $\bar\rho_\nu$ and $\bar\rho_{\nu+1}$. It is easy to prove that $\gamma_t$ is still regular (some more point of not differentiability will arise because of interactions order). We aim at proving that

\begin{equation}\label{eq:ch2:tanv-d1}
\left\Vert ((v,\xi ),\eta)_t\right\Vert \leq  e^{Kt}
\left\Vert ((v,\xi),\eta
)_0\right\Vert,
\end{equation}
for an appropriate constant $K>0$.
Then uniqueness and Lipschitz continuous dependence of solutions to Cauchy problems is straightforwardly achieved passing to the limit on the WFT-ODE
approximate solutions.

\subsection{Estimates on tangent vectors to $\rho_\nu$ due to WFT}
To estimate the shifts of waves at interactions times between waves,
we can use the estimate \cite[Lemma 2.7.2]{GPbook}
(originally proved in \cite{Bress93}).

\begin{lemma}\label{lemma4.1}
Consider two waves with speeds $\lambda_1$ and $\lambda_2$
respectively, that interact together producing a wave with
speed $\lambda_3$. If the first wave is shifted by (the tangent vector of the shift) $\xi_1$ and the
second wave by  (the tangent vector of the shift) $\xi_2$, then the (tangent vector of the) shift of the resulting wave at time of interaction is given by
\bee\label{eq51}
\xi_3 = \frac{\lambda_3 - \lambda_2}{\lambda_1 - \lambda_2} \xi_1 +  \frac{\lambda_1-\lambda_3}{\lambda_1 - \lambda_2}\xi_2.
\ede
Moreover we have
\bee\label{eq52}
\Delta \rho_3 \xi_3 = \Delta \rho_1 \xi_1 + \Delta \rho_2 \xi_2,
\ede
where $\Delta \rho_i$ are the signed strengths of the corresponding waves.
\end{lemma}

\begin{proof}
The interacting waves satisfy the ODE:
\begin{equation}\label{waves}
\frac{\partial x_i}{\partial t}(t)=\lambda_i(\Delta \rho_{i}, x_i(t)-y(t)) = \frac{\triangle[\rho_i\varphi(x_i(t)-y(t))(1-\rho_i)]}{\triangle\rho_i}, \qquad i=1,2 .
\end{equation}
From the regularity of $\varphi$ we have that $x_i$ are smooth functions
of time, thus we can linearize them around the interaction time
and apply Lemma 2.7.2 of \cite{GPbook}.
\end{proof}

The wave shifts change also when the waves are located within
the influence zone of the moving bottleneck, more precisely we have:
\begin{lemma}\label{lem:wavebuszone}
Consider a wave $(\rho_l,\rho_r)$ with shift $\xi$ which is in the influence zone
of the moving bottleneck on the time interval $[t_1,t_2]$
and does not interact with other waves of the bottleneck.
Then we have:
\begin{equation}\label{stimaode-waves-enun}
|\xi (t_2)| \le |\xi (t_1)| e^{\|\varphi'\|_{\infty} (t_2-t_1)}.
\end{equation}
\end{lemma}

\begin{proof}
Let $x(t;z)$, $t \in [t_1, t_2]$ be the unique solution to the ODE:
 \begin{equation}\label{ode-waves}
\left\{
\begin{array}{l}
\frac{\partial x}{\partial t} =\lambda(\rho_l, \rho_r, \zeta)\\
x(t_0)= z.
\end{array}
\right.
\end{equation}
Let $\bar{x}$ be the position of the wave at time $t_1$ and
consider the solution to shifted initial position
$x(t;\bar{x}+\xi)$, $t \in [t_1, t_2]$.
Note that, since $\lambda= \frac{\varphi(\zeta)\Delta (\rho(1-\rho))}{\Delta \rho} $, the solution $x(t;\cdot)$ satisfies
\begin{equation}\label{stimaode-waves}
|x(t;\bar{x}+\xi)-x(t,\bar{x})| \le |\xi | e^{C(t - t_1)},
\end{equation}
uniformly in $t \in [t_1, t_2]$, with $C = \|\varphi'\|_{\infty} = \max_{\mathbb{R}}|\varphi'|$
the uniform Lipschitz constant for $\lambda$,
 and then we also obtain the estimate
\begin{equation}\label{stimaode2-waves}
|\xi(t)|= \lim_{\epsilon \to 0}\left|\frac{x(t;\bar{x}+\epsilon \xi(t_1))-x(t;\bar{x})}{\epsilon}\right| \le e^{C(t-t_1)}\, |\xi(t_1)|, \quad t \in [t_1, t_2].
\end{equation}
\end{proof}

\subsection{Estimates on tangent vectors to $\rho_\nu$ due to the source}
As specified above, our WFT algorithm includes a sampling procedure on the interval $[\bar y-\beta,\bar y+\beta]$ where $\rho^{source}$ is not piecewise constant due to the action of the source term; see \eqref{updatedsourceterm}. The so-defined  piecewise constant  approximation $\rho^{updated}$ has new  tangent vectors which are estimated in the following lemma.
\begin{lemma}\label{lem:source}
Let $(\rho_l,\rho_r)$ be a discontinuity of $\rho^{source}$ with shift $\xi^-$ located inside the region of influence of the bottleneck, and let us denote with $|\Delta \rho^-|$ its strength: $|\Delta \rho^-| = | \rho_l - \rho_r|$. Then in $\rho^{updated}$ this discontinuity has  unmodified shift and a new strength  $|\Delta \rho^+|$ which verifies the following estimate
\begin{equation}\label{eq:sourcupdate}
|\Delta \rho^+| \leq |\Delta \rho^-|
(1+ \|\varphi'\|_{\infty}\, \Delta t ).
\end{equation}
As a consequence, the tangent vector is estimated as follows:
\begin{equation}\label{eq:sourcupdate2}
|\xi^+\ \Delta \rho^+|  \leq
 |\xi^- \Delta \rho^- | (1+ \|\varphi'\|_{\infty} \,  \Delta t),
\end{equation}
where $\xi^+ = \xi^-$ is the  (unmodified) shift of the wave in $\rho^{updated}$.
\end{lemma}
\begin{proof}
Due to Step 4 in our WFT algorithm (cfr.\ Section  \ref{sec:WFT-ODE}),  a wave $(\rho_l,\rho_r)$ of  $\rho^{source}$ located at a point in the region of influence of the bottleneck will be located in $\rho^{updated}$ at the same point, that is, $\xi^- = \xi^+$, indicating by ${}^\pm$ the values  before and after the source action and the sampling.
However, the jump will be changed by this procedure, namely
$(\rho_l-\Delta t\ \partial_x\varphi\ \rho_l(1-\rho_l),
\rho_r-\Delta t\ \partial_x\varphi\ \rho_r(1-\rho_r) )$. Therefore
\begin{equation*}
 | \Delta\rho^+|\leq
| \Delta \rho^- | \ (1+\Delta t \|\varphi'\|_{\infty}
\|\partial_\rho (\rho(1-\rho))\|_{\infty})
\leq |\Delta \rho^- | \ (1+ \|\varphi'\|_{\infty} \,  \Delta t)
\end{equation*}
and in addition
\begin{equation*}
|\xi^+\ \Delta \rho^+| =| \xi^- \ \Delta\rho^+|\leq
|\xi^- \Delta \rho^- | \ (1+\Delta t \|\varphi'\|_{\infty}
\|\partial_\rho (\rho(1-\rho))\|_{\infty})
\leq |\xi^- \Delta \rho^- | \ (1+ \\|\varphi'\|_{\infty}\,  \Delta t).
\end{equation*}
\end{proof}

\subsection{Estimates on tangent vectors to $y_\nu$}
The tangent vector to $y_\nu$, i.e. $\eta$, changes only
at interaction times of the moving bottleneck with waves.
Both the moving bottleneck and the waves satisfy ODEs
with smooth right-hand-side, thus by linearization we find
the same interaction estimate as for the case treated in \cite{BP08}.

\begin{proposition}\label{stimaetapiu}
Let $t^*$ be an interaction time between the moving bottleneck and
a wave $(\rm,\rp)$ and indicate by $\eta_\pm$ the moving bottleneck
shift after and before the interaction respectively. Then we have:
\begin{align}
&|\eta_+| \le |\eta_-| +\frac{ w_{max}}{\mu}\, |\xi_-|\, |\rp-\rm|, & &\hbox{if the wave is a shock};
\label{stima_eta+s}\\
&|\eta_+| \le |\eta_-| \left(1 + \frac{w_{max}\delta_\nu}{\mu}\right)+
\frac{ w_{max}}{\mu}\, |\xi_-|\, |\rp^\theta-\rm^\theta|, & &\hbox{if the wave is a rarefaction.}
\label{stima_eta+}
\end{align}
where $\mu$ is given by Lemma \ref{lemma3p1}.
\end{proposition}

\subsection{Convergence of WFT-ODE}
To prove convergence of the WFT-ODE scheme, we estimate the increase of the tangent vectors $\eta$
to the moving bottleneck position $y_\nu$. We fix $\Delta t$ and $\nu$ and start introducing some notation.\\
\begin{definition}
We denote by $(TV)_j^{p}$ the total variation of waves of $\rho_\nu(j\Delta t)$ which are to the right of the moving
bottleneck (thus can potentially interact with it), and by $(TV)_j^{i}$ the total variation of waves
of $\rho_\nu$ which interact with the moving bottleneck in the time interval $[j\Delta t, (j+1) \Delta t]$.\\
Similarly we denote by $V_j^{p}$ the sum of the norms of tangent vectors to waves of $\rho_\nu(j\Delta t)$
which are to the right of the moving bottleneck (thus can potentially interact with it),
and by $V_j^{i}$ the norms of tangent vectors to waves of
$\rho_\nu$ which interact with the moving bottleneck in the time interval $[j\Delta t, (j+1) \Delta t]$.\\
We denote by $\eta_j$ the tangent vector to $y_\nu$ at time $j\Delta t$.
\end{definition}

With these notations we have the following estimates. From Lemma \ref{le:step3-tv-est}, we get:
\begin{equation}\label{eq:est:TVj}
(TV)_{j+1}^p\leq (TV)_{j}^p - (TV)_{j}^{i} + K_\varphi \Delta t.
\end{equation}
From Lemmas \ref{le:step3-tv-est}, \ref{lemma4.1}, \ref{lem:wavebuszone} and \ref{lem:source},
and Remark \ref{rem:shftsb},  we get:
\begin{equation}\label{eq:est:Vj}
V_{j+1}^p\leq e^{\|\varphi'\|_{\infty} \Delta t} (1+\|\varphi'\|_{\infty} \Delta t)\ V_j^p +
K_\varphi \Delta t\, \eta_{j+1} - V_j^i.
\end{equation}
More precisely, Lemma \ref{le:step3-tv-est} guarantees that the total variation of new waves due to Step 3
is bounded by $K_\varphi \Delta t$, thus the new contribution to the tangent vector norm
is bounded by $K_\varphi \Delta t\, \eta_{j+1}$ by Remark \ref{rem:shftsb};
Lemma \ref{lemma4.1} guarantees that the norm of tangent vector does not increase for
waves interactions; Lemma \ref{lem:wavebuszone} shows that the magnification due to
permanence in the moving bottleneck influence zone is bounded by $e^{\|\varphi'\|_{\infty} \Delta t}$;
Lemma \ref{lem:source} provides the bound $(1+\|\varphi'\|_{\infty} \Delta t)$ on magnification due to the source term.\\
Finally, from Proposition \ref{stimaetapiu}, $\eta$ may increase
by interacting with a wave $(\rm,\rp)$ having shift $\xi$
with multiplicative factor at most
$(1+ \frac{w_{max}}{\mu} |\rm-\rp|)$ for any
rarefaction wave and additive factor
$ \frac{w_{max}}{\mu} |\xi|\ |\rm-\rp|$ for any wave.
The worst case scenario happens when first all shocks interact and then
all rarefactions, this is bounded by the estimate:
\begin{equation}\label{eq:est:etaj}
\eta_{j+1}\leq e^{\frac{w_{max}}{\mu} (TV)_j^i}\
\left(\eta_j + \frac{w_{max}}{\mu}\, V_j^i\right)
\end{equation}
We can rewrite the estimates (\ref{eq:est:TVj})--(\ref{eq:est:etaj}) as:
\begin{equation}\label{eq:est:K}
\begin{aligned}
(TV)_{j+1}^p & \leq (TV)_{j}^p - (TV)_{j}^{i} + K \Delta t,\\
V_{j+1}^p & \leq (1+K \Delta t)\ V_j^p +  K \Delta t\, \eta_{j+1} - V_j^i + o(\Delta t),\\
\eta_{j+1} & \leq  e^{K (TV)_j^i}\, (\eta_j + K\, V_j^i),
\end{aligned}
\end{equation}
where $K=\max \{K_\varphi, 2\|\varphi'\|_{\infty} , \frac{w_{max}}{\mu}  \}$.
We now have the following:
\begin{lemma}
Consider sequences $(TV)_j^p$, $(TV)_j^i$, $V_j^p$, $V_j^i$ and $\eta_j$ satisfying (\ref{eq:est:K}).
Let $\mu_j$ be the solution to the system:
\begin{equation}\label{eq:est:K'}
\begin{aligned}
V_{j+1} & =   K \Delta t\, \mu_{j+1} + o(\Delta t), \\
\mu_{j+1} & = \mu_j + K\, e^{KT}\, V_j,
\end{aligned}
\end{equation}
with initial datum $\mu_1= K e^{KT}\, V_0$ and $V_1=K\Delta t\, \mu_1$.
Then we have $\eta_j\leq e^{K\, TV(\bar{\rho})+KT} \mu_j$.
\end{lemma}
\begin{proof}
First we have $\sum_i (TV)_j^i\leq TV(\bar{\rho})+KT$.
Notice that $\eta_j$ is always increasing so the worst case
for the multiplicative terms $e^{K (TV)_j^i}$ in \eqref{eq:est:K}
(third equation)
is when the $(TV)_j^i=0$ for $j=1,\ldots, \frac{T}{\Delta t}-1$
and $(TV)_{\frac{T}{\Delta t}}^i=TV(\bar{\rho})+KT$.
Also notice that the increase of $\eta_{j+1}$ due to the $V_j^i$ term
is maximized when $V_j^i=V_j^p$. On the other side
this may not achieve the maximal increase in $V^p_{j+1}$
because of the multiplicative term $(1+K\Delta t)$.
However, such maximal increase of $V^p_{j+1}$  on the interval $[0,T]$ (thus after all time steps) is bounded by $e^{KT}$,
which is the term appearing on the right-hand side of (\ref{eq:est:K'})
(second equation).
Therefore, with this corrected multiplicative term, the increase is maximized.
Finally, $\eta_0=0$ thus the initial data for \eqref{eq:est:K'}
are given by $\mu_1= K e^{KT}\, V_0$ and $V_1=K\Delta t\, \mu_1$.
\end{proof}

We are now ready to estimate $\eta_j$. First notice that we have:
\[
\mu_{j+1}\leq \mu_j + K\, e^{KT}\, V_j \leq
(1+K^2 e^{KT} \Delta t)\ \mu_j
\]
thus:
\[
\mu_{j}\leq e^{K^2T\,  e^{KT}} (K e^{KT}\, V_0)
\]
and finally:
\begin{equation}\label{eq:est:etafinal}
\eta_j\leq e^{K\, TV(\bar{\rho})+KT} \mu_j
\leq  e^{K\, TV(\bar{\rho})+KT} e^{K^2T\, e^{KT}} (K e^{KT}\, V_0)
\end{equation}

Thus we obtain the following:
\begin{proposition}
Consider initial conditions  $\bar\rho \in \mathcal{BV}(\R)$, $y(0)\in\R$ and time horizon $T>0$.
Using the notation of Lemma \ref{le:T-est}, assume that $\eta<1$ and $T$ satisfy (\ref{eq:cond:T}).
Let $(\rho_\nu,y_\nu)$ be the approximate solutions
computed by the WFT-ODE scheme, then:
\begin{equation}
|y_{\nu+1}(t)-y_\nu(t)| \le  K_1\   2^{-(\nu+1)}.
\end{equation}
where $K_1$ depends only on the total variation of $\bar\rho$, $T$,
$w_{max}$,
$\|\varphi\|_{\infty}$, $\|\varphi'\|_{\infty}$, $\mu$ defined in (\ref{eq:def-mu})
and $K_\varphi$ defined in (\ref{additionalTVduetosourceterm}).
\end{proposition}
\begin{proof}
From (\ref{eq:est:etafinal}), we have that $\eta_\nu$ can be estimated in terms of the total variation of $\bar{\rho}$, $T$, the constant $K$ and $V_0$.
But $V_0$ is estimated by the total variation of $\bar{\rho}$
and the initial shifts, thus we conclude by
Lemma \ref{lem:initialshifts}.
\end{proof}

\section{Godunov-ODE-FS scheme}\label{sec:GOF}

Here we introduce a numerical scheme called Godunov-ODE-FS (GOF),
which is based on fractional step method combining Godunov scheme for the PDE and
exact solution for the ODE. The dynamics of (\ref{lwr}) and (\ref{ode}) are thus solved separately at each iteration.
The use of Godunov scheme is motivated by its easy implementation and its connection with the modelling
of vehicular traffic problems, see \cite{GPbook,Leb}.

\subsection{Godunov scheme for the PDE}\label{Godunov_ode}
Following the ideas described in Section \ref{WFT} for the WFT algorithm, we report in the following the modified Godunov scheme for a conservation law with source term
\begin{equation}\label{lc+st_bis}
\partial_t \rho + \partial_\rho g ~\partial_x \rho = - \partial_x g
\end{equation}
(see (\ref{lc+st})).
We first introduce a numerical grid, denoting by $\Delta x$ the space mesh size and by $\Delta t$
the time mesh size. Moreover, we denote by $(t_l,x_m)=(l \Delta t,m \Delta x)$ the grid points for
$l=0,1,\dots,L$, $m=0,1,\dots,M$, where $L$ and $M$ are, respectively, the number of time and space nodes of the
grid, and by $C^l_m$ the discretization cell $(t_l,t_{l+1}) \times (x_{m-1}, x_m)$.
For a function $u$ defined on the grid we write $u_m^l=u(t_l,x_m)$.

Let us denote by
\[
W\left(\frac{x-x_{m-\frac{1}{2}}}{\Delta t};\rho_\ell,\rho_r\right)
\]
the self-similar entropy solution of the unique Riemann problem defined on $C^l_m$ (with discontinuity point $x_{m-\frac{1}{2}}$) and let us define
the numerical flux $F$ as $F(\rho_\ell,\rho_r,t,x)=g(W(0;\rho_\ell,\rho_r),t,x)$.

First, we replace the initial datum $\bar\rho(x)$ by a piecewise constant approximation,
\[
\rho^{0}_m = \frac{1}{\Delta x} \int_{x_{m-\frac{1}{2}}}^{x_{m+\frac{1}{2}}} \bar\rho(x) d x.
\]
Then, we alternate a single step of the classical Godunov scheme 
$$
\rho^{*}_{m}=\rho^{l}_{m}-\frac{\Delta t}{\Delta x}\Big(F(\rho^{l}_{m},\rho^{l}_{m+1},t_l,x_{m+\frac 1 2})-F(\rho^{l}_{m-1},\rho^{l}_{m},t_l,x_{m-\frac 1 2})\Big)
$$
with a single step of the Euler scheme
$$
\rho^{l+1}_{m}=\rho^{*}_m+\Delta t (-\partial_x g(\rho^{*}_m,t_l,x_m)).
$$

\subsection{GOF scheme}\label{GOF}
Before describing the GOF scheme, we point out that shock waves solutions to (\ref{lwr})
have velocities depending on the moving bottleneck position $y$.
The modified Godunov scheme can be thus used, provided that, in each discretization cell, the shock speed
$\lambda(\rp,\rm, \zeta)$ defined in (\ref{rh}) does not change sign as a function of $\zeta$.
This is established by next Lemma.
\begin{lemma}\label{lemma5}
Let $\bar f (\rho) = \rho(1-\rho)$ and $\rho_\ell$ (resp., $\rho_r$) the left (resp., the right) state
of a discontinuity. Set $\bar \lambda = \frac{\bar f(\rho_r) - \bar f(\rho_\ell)}{\rho_r - \rho_\ell}$. Then,
\bee\label{segno}
sgn(\lambda(\rp,\rm,\zeta)) = sgn(\bar \lambda).
\ede
\end{lemma}

\begin{proof}
Since we have:
\begin{eqnarray*}
\lambda(\rp,\rm, \zeta) &=& \frac{f(\rho_r,\zeta) - f(\rho_\ell,\zeta)}{\rho_r-\rho_\ell} = \varphi(\zeta) \frac{\bar f(\rho_r) - \bar f(\rho_\ell)}{\rho_r - \rho_\ell} =\\ &=& \bar \lambda \cdot \varphi(\zeta),
\end{eqnarray*}
 the result (\ref{segno}) is easily achieved noticing that $\varphi(\zeta) > 0$.
\end{proof}
The situation described by Lemma \ref{lemma5} is depicted in Figure\ \ref{fig:per_lemma_5}.

\begin{figure}
\begin{center}
\includegraphics[height=5cm,width=10cm]{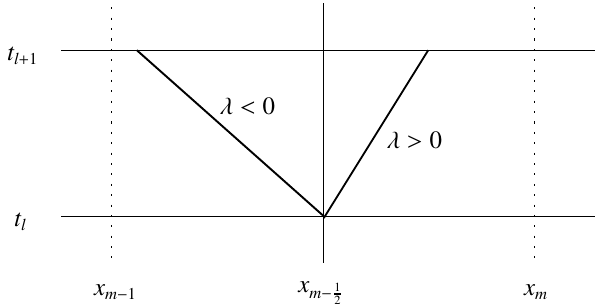}
\caption{Discontinuities have speed sign independent of the moving bottleneck position.
Then, in each discretization cell, they are either always positive or always negative.}
\label{fig:per_lemma_5}
\end{center}
\end{figure}

The discretization grid is defined as for the Godunov scheme and the value of density on grid points is called $\rho^l_m$.
The moving bottleneck position will be defined for times $t_l$ and denoted by $y^l$.
\begin{itemize}
\item[Step 0.] We approximate the initial datum $\bar\rho$ as for the Godunov scheme and set $y^0=y(0)$.
\item[Step 1.]
We assume that $y^l$ and $\rho^l_m$ were defined by previous step.\\
We keep fixed the position of the moving bottleneck, thus set $y(s) = y^l$ for $s \in [t_l, t_{l+1}]$.
Define $W_m=W_m(t,x)$ to be the self similar solution to the Riemann problem defined in $x_{m-\frac{1}{2}}$
(for the bottleneck position fixed).
This solution is formed by a wave, whose velocity varies in time but mantains its sign constant as proved by Lemma \ref{lemma5}. We can use Gauss-Green formula as for Godunov scheme and define
\begin{eqnarray*} \label{GOD-FS}
 \rho^{l+1}_m &=& \rho^l_m - \frac{1}{\Delta x} \int_{t_l}^{t_{l+1}}
 \Big[f\left(W_{m+1}\left(s,x_{m+\frac12}\right),x_{m+\frac12}-y^l\right)- f\left(W_{m}\left(s,x_{m-\frac12}\right),x_{m-\frac12}-y^l\right)\Big] \ ds.
 \end{eqnarray*}
 \item[Step 2.]
There exists $m$ such that $y^l\in [x_{m-1-\frac12},x_{m-\frac12})$, see Figure\ \ref{fig:stepB}.
	\begin{figure}
	\begin{center}
	\includegraphics[height=5cm,width=10cm]{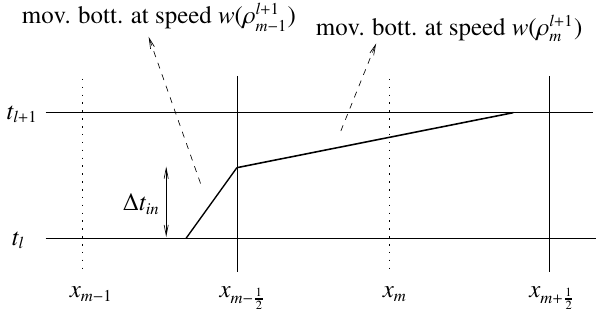}
	\vspace{.5cm}
	\caption{Iterative Step 2 of the GOF algorithm.}
	\label{fig:stepB}
	\end{center}
	\end{figure}
Using the computed density $\rho^{l+1}$,
the initial velocity of the moving bottleneck is $w(\rho^{l+1}_{m-1})$ up to the interaction time
with the line $x=x_{m-\frac12}$ then it travels with velocity $w(\rho^{l+1}_{m})$. More precisely,
we compute the interaction time as
\[
\Delta t_{in}=\frac{x_{m-\frac12}-y^l}{w(\rho^{l+1}_{m-1})}.
\]
Now, if $\Delta t_{in}\leq\Delta t$ we set
\[
y^{l+1}=x_{m-\frac12}+(\Delta t-\Delta t_{in})\ w(\rho^{l+1}_{m}),
\]
otherwise we set
\[
y^{l+1}=y^l+\Delta t\ w(\rho^{l+1}_{m-1}).
\]
\end{itemize}

\begin{remark}
Notice that for the flux (\ref{flusso}) the CFL condition reads
  \bee \label{CFL2}
  \Delta t \ \sup_{m,l} \left\{\sup_{\rho \in
 I(\rho^l_{m-1/2},\rho^l_{m+1/2})} \left|\frac{\partial f}{\partial \rho}(\rho,\zeta^l_m)\right| \right\}\le \frac{\Delta x}{2},
 \ede
where $\zeta^l_m = x_m - y^l$ and $I(\rho^l_{m-1/2},\rho^l_{m+1/2})$ is the interval with endpoints $\rho^l_{m-1/2}$ and $\rho^l_{m+1/2}$.
In particular, CFL is always satisfied if $\Delta t \leq \frac{1}{2\bar v}\Delta x$.
\end{remark}

\section{Multiple bottlenecks}\label{sec:multi}
In this section we introduce a model describing the evolution of
car traffic flow along a road in presence of $P>1$
moving bottlenecks. In this case, we face
the event that two slower vehicles can occupy the same  position,
which could lead to a non realistic  dropping of flux capacity.
To solve this problem, we introduce two modelling choices:
\begin{itemize}
\item[(A)] Two moving bottlenecks can overtake each other, thus we re-define the flux function.
\item[(B)] Two moving bottlenecks cannot overtake each other.
To achieve that, we re-define the velocity function of moving
bottlenecks.
\end{itemize}

\subsection{Model A}
Let us consider the following PDE--ODE coupled system
\begin{equation*}
	\begin{cases}
		\partial_{t}\rho+\partial_{x}f(\rho,x,y_1(t),\ldots,y_P(t))=0 &
		t>0,\ x\in\R,\ P\in\N\\
		\rho(0,x)=\bar\rho(x)& x\in\R \\
		 \dot{y_i}(t)=w_i(\rho(t,y_i))&
		 i=1,\ldots,P \\
		 y_i(0)=y_{i,0} & i=1,\ldots,P,
	\end{cases}
\end{equation*}
where $y_i=y_i(t)$ is the position of the $i$-th moving bottleneck, and $\rho=\rho(t,x)$ is the density on the road.
The  flux function $f$ is given by
	\begin{equation}\label{def::flux:multi}
	   f(\rho,x,y_1(t),\ldots,y_P(t))=
	   \rho\cdot v(\rho(t,x),x,\mathbf{y}),
	\end{equation}
where $\mathbf{y}=(y_1,\ldots,y_P)$ and the smooth function $v(x,\rho,\mathbf{y})$ is defined as
	\begin{equation}\label{def::vel}
	 v(\rho,x,\mathbf{y})=\hat v(\rho)\Phi(x,\mathbf{y}), 
	\end{equation}
	\begin{equation*}
	\Phi(x,\mathbf{y}(t))=
	\min\{\varphi_1(x-y_1),\ldots,\varphi_P(x-y_P)\},
	\end{equation*}
with the average velocity of cars $\hat v(\rho)=1-\rho$. Each function $\varphi_i(\zeta)$
takes into account the flux capacity drop due to the
presence of the $i$-th bottleneck and we assume:
\[
0<\vmin_i\leq \varphi_i(\zeta)\leq \vmax,\qquad
\varphi_i(0) = \vmin_{i},
\]
\[
\vmax=\max_{\zeta\in\R}\varphi_i(\zeta)=
\varphi_i(\zeta)\ \ \hbox{ for  every}\ \zeta\in\R\setminus
[-\beta_i,\beta_i]\ \ \hbox{ and for every}\ i=1,\ldots,P.
\]
Speeds of moving bottlenecks are defined as
$w_i(\rho)=w_{max}^i(1-\rho)$ with $\vmin_{i}> w_{max}^i$.

Now we admit the presence of many slow-moving vehicles in the car traffic flow,
with possibly different characteristics among each other.
Therefore, in general, functions $w_{i}(\cdot)$ and $\varphi_{i}(\cdot)$ are different
and $P$ can be  arbitrarily large.

In view of our definition of the function $\Phi$ in the flux,
when two slow vehicles are far enough, we go back to the case of a single moving bottleneck.
On the other side, if two slow vehicles occupy the same position, the flux capacity drop is the largest among the two.

\subsection{Model B}
We now introduce the alternative modelling choice in (B), which avoid overtaking between slow vehicles,
again described by a coupled PDE-ODE system
\begin{equation}\label{statement}
	\begin{cases}
		\partial_{t}\rho+\partial_x f(\rho, x,y_1(t),\ldots,y_P(t))=0 &
		t>0,\ x\in\R,\ P\in\N\\
		\rho(0,x)=\bar\rho(x)& x\in\R \\
		 \dot{y}_i(t)=w_i(\rho(t,y_i),  y_{i}, y_{i+1})&
		 i=1,\ldots,P-1 \\
		 \dot{y}_P(t)=w_P(\rho(t,y_P))&  \\
		 y_i(0)=y_{i,0} & i=1,\ldots,P,
	\end{cases}
\end{equation}
with $y_{1,0}<\ldots<y_{P,0}$. The flux function is defined as in
\eqref{def::flux:multi}--\eqref{def::vel}, but this time the function $\Phi$ is given by
	\begin{equation}\label{def::cutoff}
	\Phi(x,\mathbf{y})= \prod_{i=1}^P\varphi_i(x-y_i(t)).
	\end{equation}
We define differently also the velocities of moving bottlenecks. More precisely,
to avoid overtaking among them, and indeed to avoid the superposition of the zones where functions $\varphi_i$
take values less than $\vmax$, we suitable define the velocity functions $w_i(\rho,y_i,y_{i+1})$.
First let $\omega_i=\omega^i_{max}(1-\rho)$, with $0<\omega^i_{max}<\vmin_i$, and then back recursively on $i$ set
	\begin{equation}\label{eq:defvelbuses}
	w_i(\rho, y_{i}, y_{i+1})
		=\begin{cases}
		\omega_{i}(\rho), & \hbox{if}\ y_{i+1} - y_{i} \geq 2(\beta_{i+1} +
		\beta_{i})\\
		 \min\{\omega_i(\rho), w_{i+1}(\rho, y_{i+1}, y_{i+2})\}, & \hbox{if}\
		 y_{i+1} - y_{i} \leq \beta_{i+1} + \beta_{i}
	\end{cases}
	\end{equation}
and complete the definition for $y_{i+1} - y_{i}\in [\beta_{i+1} + \beta_{i},2(\beta_{i+1} + \beta_{i})]$
by smooth monotone interpolation.
This choice introduces a follow-the-leader flavour in the microscopic dynamic and avoids the modelling problems described above.
A bottleneck at position $y_{i}$ moves with velocity $\omega_{i}(\rho)$,
depending only on the density of cars along the road, provided the distance
from the vehicle ahead in position $y_{i+1}$ is sufficiently large, namely, larger
than $2(\beta_{i+1} + \beta_{i})$.  If such distance decreases, then $y_i$ starts to decelerate and
eventually move with the same velocity as the vehicle ahead when
$y_{i+1} - y_{i} = \beta_{i+1}+\beta_{i}$.

\subsection{GOF algorithm for multiple bottlenecks}\label{sec:algorithm}

We introduce a Godunov-ODE-FS algorithm for model (B), while model (A) can be treated in a manner completely
similar to the case of a single bottleneck. Notice that for model (B) slow moving vehicles can influence each other,
thus the numerics is more involved.

The discretization grid is defined as for the Godunov scheme and the value of density on grid points is called $\rho^l_m$.
The moving bottlenecks' positions will be defined for times $t_l$ and denoted by $y_i^l$, $i=1,\ldots,P$.
For consistency we suppose
	\begin{equation*}
 	y_1(0)<y_2(0)-(\beta_2+\beta_1)<\ldots<y_P(0)-(\beta_P+\beta_{P-1}).
 	\end{equation*}
Notice that the $P$-th bottleneck plays the role of the leader and it is not influenced by positions of other slow vehicles.
Thus its trajectory can be traced  as in Section \ref{GOF}.
Furthermore, because of definition of $w_i$,
the $i$-th bottleneck's trajectory is only influenced by the $(i+1)$-th bottleneck trajectory.
More precisely, in the scheme we let the $i$-th bottleneck proceed with velocity $\omega_i$ as long as the distance with the
$(i+1)$-th bottleneck is larger than $\beta_{i+1}+\beta_{i}$, otherwise we let it proceed with the velocity of the $(i+1)$-th bottleneck.\\
Let us first note that, as for the case of a single bottleneck,
the velocity of any wave inside each cell change in time, but its sign does not. Indeed, let $\lambda(\rho_r,\rho_\ell,x,\mathbf{y})$ be the speed of the wave, then, using notation of Lemma \ref{lemma5}, we have
	\begin{align*}
	\lambda(\rho_r,\rho_\ell,x,\mathbf{y})&=
	\frac{f(\rho_r,x,\mathbf{y})-f(\rho_\ell,x,\mathbf{y})}{\rho_r - \rho_\ell}\\
	&=\Phi(x,\mathbf{y})\frac{\bar f(\rho_r)-\bar f(\rho_\ell)}{\rho_r-\rho_\ell}=\Phi(x,\mathbf{y})\bar\lambda(\rho_r,\rho_\ell)
	\end{align*}
and, since $\Phi(x,\mathbf{y})>0$, it follows that
\[
sgn(\lambda(\rho_r,\rho_\ell,x,\mathbf{y}))=sgn(\bar\lambda(\rho_r,\rho_\ell)).
\]
We are now ready to introduce the GOF algorithm.
\begin{itemize}
\item[Step 0.] We approximate the initial datum $\bar\rho$ as for the Godunov scheme and set $y^0_i=y_i(0)$.
\item[Step 1.]
We assume that $y_i^l$, for $i=1,\ldots,P$, and $\rho^l_m$ were defined by previous step.\\
We keep fixed the positions of moving bottlenecks, thus set $y_i(s) = y_i^l$ for $s \in [t_l, t_{l+1}]$.
Let $W_m=W_m(t,x)$ be the self similar solution to the Riemann problem defined in $x_{m-\frac{1}{2}}$
(for bottlenecks' positions fixed.)
This solution is formed by waves, whose velocity sign is constant.
Setting $\mathbf{y}^l=\mathbf{y}(t_l)$ and using the Gauss-Green formula, the scheme is expressed in the integral formulation as

	\begin{equation*}
		  \rho_m^{l+1}=\rho_m^{l}-\frac1{\Delta x}\int_{t_l}^{t_{l+1}}\left[
		  f\left(W_{m+1}\left(s,x_{m+\frac12}\right),x_{m+\frac12},\mathbf{y}^l  \right)- f\left(W_m\left(s,x_{m-\frac12}\right),x_{m-\frac12},\mathbf{y}^l \right)   \right] \, ds.
		\end{equation*}
\item[Step 2.]  We compute $y^{l+1}_P$ using the density $\rho^{l+1}$ obtained at Step 1, as for the case of a single bottleneck.
More precisely, there exists $m$ such that $y^l_P\in [x_{m-1-\frac12},x_{m-\frac12})$; $y_P$
moves with velocity $\omega_P(\rho^{l+1}_{m-1})$ until the interaction time $\Delta t_{in}$
with the line $x=x_{m-\frac12}$, provided that $\Delta t_{in}\leq\Delta t$.
After that time, $y_P$ moves with speed $\omega_P(\rho^{l+1}_{m})$.
Finally, if $\Delta t\leq\Delta t_{in}$ then we set
\[
y^{l+1}_P=y^l_P+\Delta t\ \omega_P(\rho^{l+1}_{m-1}),
\]
otherwise
\[
y^{l+1}_P=x_{m-\frac12}+(\Delta t-\Delta t_{in})\ \omega_P(\rho^{l+1}_{m}).
\]
\item[Step 3.](Figure 7)
We compute $y^{l+1}_i$ by backward recursion on $i=P-1,\ldots,1$.
More precisely, we define a trajectory $y_i(t)$ for every $t\in [t_l,t_{l+1}]$ and
set $y^{l+1}_i=y_i(t_{l+1})$.\\
First for $i=P$, if $\Delta t\leq\Delta t_{in}$ we set
\[
y_P(t)=y^l_P+(t-t_l)\ \omega_P(\rho^{l+1}_{m-1}),
\]
otherwise
\[
y^{l+1}_P=\left\{
\begin{array}{ll}
y^l_P+(t-t_l)\ \omega_P(\rho^{l+1}_{m-1}) & t_l\leq t\leq t_l+\Delta t_{in}\\
x_{m-\frac12}+(t-\Delta t_{in})\ \omega_P(\rho^{l+1}_{m})
& t_l+\Delta t_{in}< t\leq t_{l+1}.\\
\end{array}\right.
\]
Now, fix $i$ and assume to have defined $y_j(t)$ for $j\geq i+1$ and $t\in [t_l,t_{l+1}]$.
Let $m$ be such that $y^l_i\in [x_{m-1-\frac12},x_{m-\frac12})$ and define $\Delta t_{in}$ the time at which the line
$y^l_i+(t-t_l)\omega_i(\rho^{l+1}_{m-1})$ intersects the vertical line $x=x_{m-\frac12}$.
If $\Delta t\leq\Delta t_{in}$ we set
\[
\tilde{y}_i(t)=y^l_i+(t-t_l)\ \omega_i(\rho^{l+1}_{m-1}),
\]
otherwise
\[
\tilde{y}^{l+1}_i=\left\{
\begin{array}{ll}
y^l_i+(t-t_l)\ \omega_i(\rho^{l+1}_{m-1}) & t_l\leq t\leq t_l+\Delta t_{in}\\
x_{m-\frac12}+(t-\Delta t_{in})\ \omega_i(\rho^{l+1}_{m})
& t_l+\Delta t_{in}< t\leq t_{l+1}.
\end{array}\right.
\]
Now, if there exists a time $\Delta t_{b}< \Delta t$ such that
$y_{i+1}(t_l+\Delta t_{b})-\tilde{y}_{i}(t_l+\Delta t_{b})=\beta_{i+1}+\beta_i$, then we set
\[
{y}^{l+1}_i=\left\{
\begin{array}{ll}
\tilde{y}_i(t) & t_l\leq t\leq t_l+\Delta t_{b}\\
y_{i+1}(t)-(\beta_{i+1}+\beta_i) & t_l+\Delta t_{b}< t\leq t_{l+1},
\end{array}\right.
\]
otherwise we simply set $y_i(t)=\tilde{y}_i(t)$ for every $t\in [t_l,t_{l+1}]$.
\end{itemize}

\begin{figure}[h!]
    \begin{center}
	\includegraphics[]{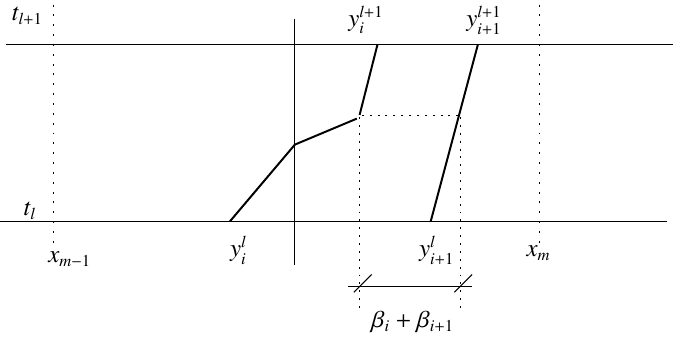}
	\caption{Iterative Step 3 of the GOF algorithm. The bottleneck
	behind takes the speed of the one ahead when the distance
	is equal to $\beta_{i}+\beta_{i+1}$.}
	\label{fig:stepB2}
    \end{center}
\end{figure}

\begin{remark}
Notice that for the dynamics (\ref{statement}) the CFL condition reads
  \bee \label{CFL3}
  \Delta t \ \sup_{m,l} \left\{\sup_{\rho \in
 I(\rho^l_{m-1/2},\rho^l_{m+1/2})} \left|\frac{\partial f}{\partial \rho}(x_m,\mathbf{y}^l_m,\rho)\right| \right\}\le \frac{\Delta x}{2},
 \ede
where $I(\rho^l_{m-1/2},\rho^l_{m+1/2})$ is the interval with endpoints $\rho^l_{m-1/2}$ and $\rho^l_{m+1/2}$.
In particular, CFL is always satisfied if $\Delta t \leq \frac{1}{2\max_i\vmax_i}\Delta x$.
\end{remark}

\section{Numerical tests}\label{sec:simulations}

In this section we show some simulations obtained by models \eqref{lwr}--\eqref{ode} and  \eqref{statement}.
In particular, we focus on interactions of bottlenecks with different type of waves, namely shocks and rarefactions.
For both models we implement the corresponding GOF scheme.

\subsection{Model \eqref{lwr}--\eqref{ode}}
To compute the approximate solution to \eqref{lwr}--\eqref{ode} we use the numerical algorithm presented
in Section \ref{GOF}.
We set the space and the time mesh size of the grid to be $\Delta x=0.02$ and $\Delta t=0.01$.
Setting parameters, we assume $w(\rho)=w_{max}(1-\rho)$ with $w_{max}=0.4$ and
	\begin{equation*}
	\varphi(\zeta)=\left\{
	\begin{array}{ll}
	\vmax-(\vmax-\vmin)\exp\left\{\ -\frac{\zeta^2}{\beta-|\zeta|}\right\} &
	 \textrm{if}\ |\zeta|\leq\beta\\
	\vmax & \textrm{otherwise},
	\end{array}\right.
	\end{equation*}
where $\vmax=1$, $\vmin=0.6$ and $\beta=0.1$.

\begin{figure}[ht!]
\begin{minipage}[c]{0.5\textwidth}
\begin{center}
\fbox{
\includegraphics[height=\textwidth]{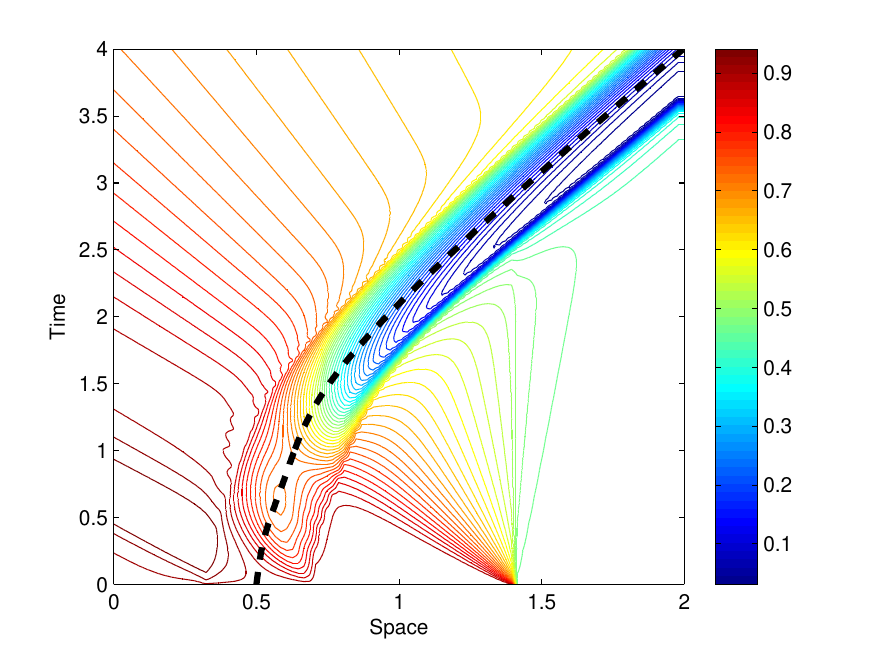}}
\end{center}
\end{minipage}
\hspace{30mm}
\begin{tabular}{ll}
\hline
\multicolumn{2}{c}{\bf initial density $\bar\rho(x)$}\\
\hline
$\rho_\ell=0.9$& if $x\leq1.4$ \\
$\rho_r=0.45$& if $x>1.4$\\
\hline
\hline
\multicolumn{2}{c}{\bf initial position }\\
\hline
$y_0=0.5$ \\
\hline
\end{tabular}
\caption{A bottleneck interacting with a rarefaction wave.}
\label{figure:oneRar}
\end{figure}

In Figure\ \ref{figure:oneRar} a moving bottleneck starting from $y_0=0.5$ interacts with
a rarefaction wave centred in $x=1.4$, generated by the Riemann datum $\rho_\ell=0.9$ for  $x\leq1.4$ and $\rho_r=0.45$ for $x>1.4$.  The vehicle is initially moving with low speed ($w=0.04$) due to the high density until it interacts
with the first characteristic of the rarefaction;
after that it accelerates since the density is decreasing.

\begin{figure}[hb!]
		\begin{tabular}{ll}
        	\hline
        	\multicolumn{2}{c}{\bf initial density $\bar\rho(x)$}\\
     		\hline
       	 	$\rho_\ell=0.3$& if $x\leq1.4$ \\
        	$\rho_r=0.9$& if $x>1.4$\\
        	\hline
	    	\hline
	        \multicolumn{2}{c}{\bf initial position }\\
     	        \hline
     	         $y_0=0.5$ \\
     	        \hline
        	\end{tabular}
        \hspace{1mm}
		\begin{minipage}[c]{0.5\textwidth}
		\begin{center}
		\fbox{
   		\includegraphics[height=\textwidth]{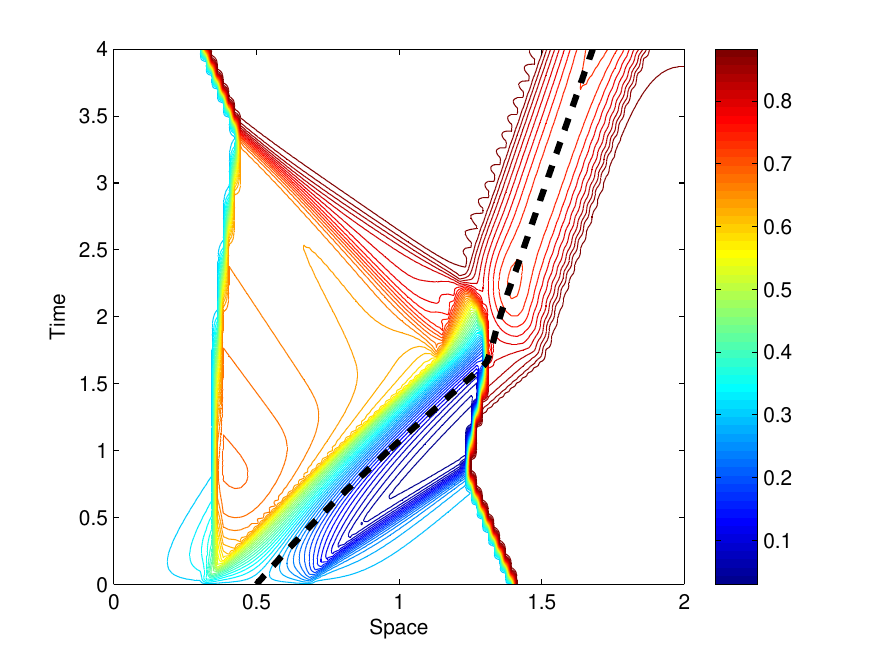}}
		\end{center}
		\end{minipage}
		\caption{A bottleneck interacting with a shock.}
		\label{figure:oneShock}
\end{figure}

In Figure \ref{figure:oneShock} a bottleneck starting from $y_0=0.5$ is interacting with a backward moving shock starting from $x=1.4$.
The latter is generated by the Riemann datum $\rho_\ell=0.3$ for  $x\leq1.4$ and $\rho_r=0.9$ for $x>1.4$, so that the speed of the shock,
before it interacts with the bottleneck, is $\lambda=-0.2$.
On a following time interval, during the interaction with the bottleneck
(from around $t=1$ until around $t=2.5$ in Figure\ \ref{figure:oneShock}), the speed of the shock almost vanishes,
because of the capacity drop caused by the bottleneck.

\begin{figure}[t!]
		 \begin{minipage}[c]{0.5\textwidth}
		 	\begin{center}
			 \fbox{
   			 \includegraphics[height=\textwidth]{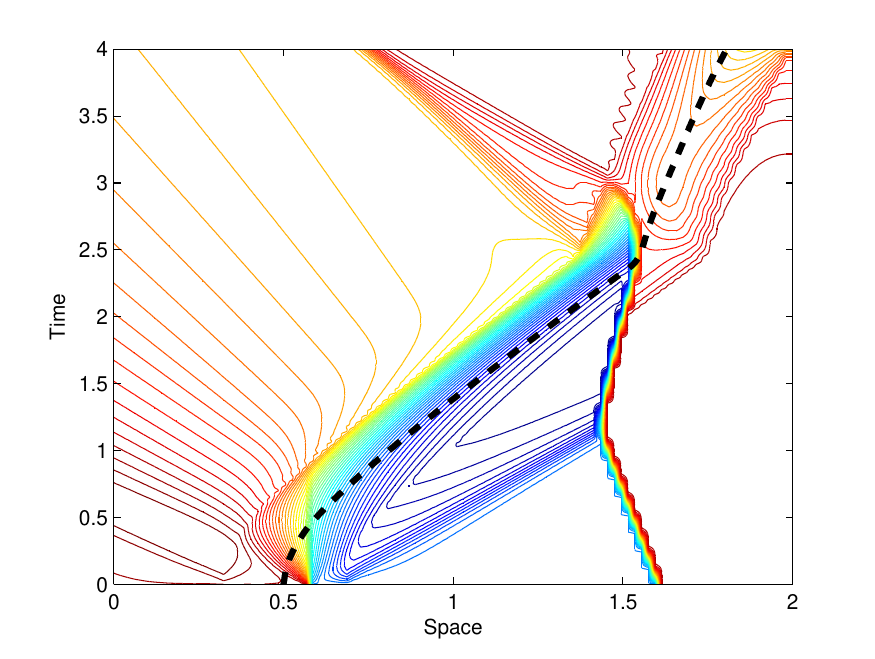}}
			 \end{center}
		 \end{minipage}
	     \hspace{30mm}
  	     \begin{tabular}{ll}
   		  \hline
   		  \multicolumn{2}{c}{\bf initial density $\bar\rho(x)$}\\
  		  \hline
   		  $0.9$ & if $x\leq 0.6$ \\
		   $0.25$ & if $0.6< x\leq 1.6$ \\
   		    $0.9$ &  if $x>1.6$\\
    		    \hline
		    \hline
	        \multicolumn{2}{c}{\bf initial position }\\
     	        	\hline
     	        	 $y_0=0.5$ \\
     	        	\hline
     	       \end{tabular}
	\caption{One bottleneck interacting both with a rarefaction and a shock.}
	\label{figure:oneBoth}
\end{figure}

Figure\ \ref{figure:oneBoth} shows a shock and a rarefaction which interact with each other
and with the bottleneck. The rarefaction wave is centred in $x=0.6$ while the shock is starting from $x=1.6$.

\subsection{Model \eqref{statement}}
Simulations for the system in \eqref{statement} are provided by means of the numerical algorithm defined in
Section \ref{sec:algorithm}. Interactions of three bottlenecks and different type of waves are considered.
We set the space and time mesh size to be $\Delta x=0.02$ and $\Delta t=0.01$.
Velocity functions for slow moving vehicles are $\omega_i(\rho)=\omega_{max}^i(1-\rho)$,
with $\omega_{max}^1=0.49$ and $\omega_{max}^i=0.4$ $i=2,3$. We set again
	\begin{equation*}
	\varphi_i(\zeta)=\left\{
	\begin{array}{ll}
	\vmax-(\vmax-\vmin_i)\exp\left\{\ -\frac{\zeta^2}{\beta_i-|\zeta|}\right\} &
	 \textrm{if}\ |\zeta|\leq\beta_i\\
	\vmax_i & \textrm{otherwise},
	\end{array}\right.
	\end{equation*}
where, as before, $\vmax=1$.
\begin{figure}[t!]
	\begin{minipage}[c]{0.5\textwidth}
		\fbox{\centering
		\includegraphics[height= \textwidth]{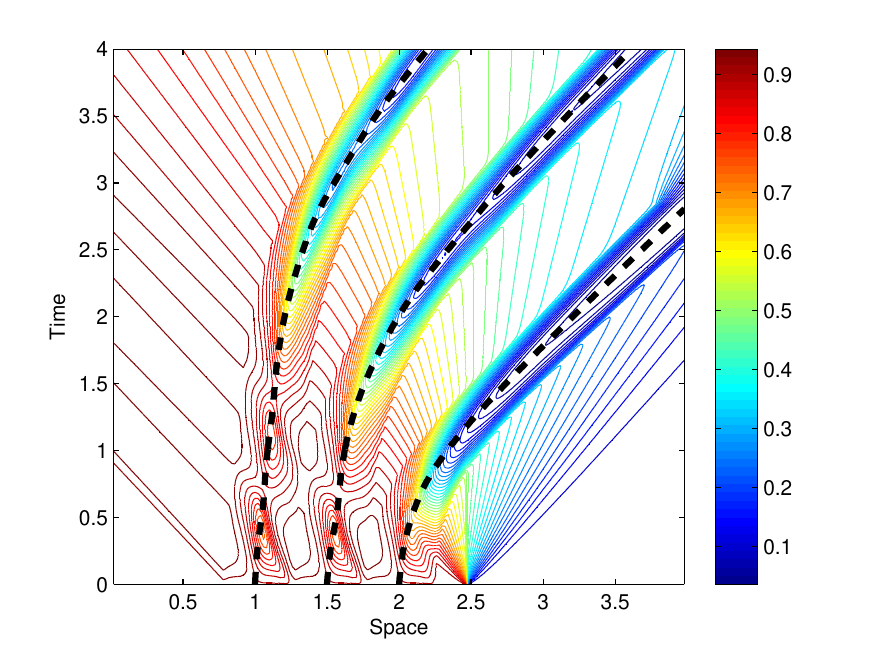}}
		\end{minipage}
		\hspace{30mm}
        \begin{tabular}{cl}
       		 \hline
       		 \multicolumn{2}{c}{\bf initial density $\bar\rho(x)$}\\
      		  \hline
      		  $\rho_\ell=0.9$& if $x\leq2.5$ \\
     		  $\rho_r=0.1$& if $x>2.5$\\
     		  \hline
		  \hline
	          \multicolumn{2}{c}{\bf initial positions }\\
     		 \hline
       	 	 $y_1(0)=$&1 \\
        	 $y_2(0)=$& 1.5\\
		 $y_3(0)=$&2\\
        	\hline
        \end{tabular}
\caption{Three bottlenecks interacting with a rarefaction wave.}
\label{threeRar}
\end{figure}
	\begin{figure}[ht!]
        	\begin{tabular}{ll}
        	\hline
        	\multicolumn{2}{c}{\bf initial density $\bar\rho(x)$}\\
     		\hline
       	 	$\rho_\ell=0.85$& if $x\leq3.5$ \\
        	$\rho_r=0.95$& if $x>3.5$\\
        	\hline
		\hline
        	\multicolumn{2}{c}{\bf initial positions }\\
     		\hline
       	 	$y_1(0)=$&1 \\
        	$y_2(0)=$& 2\\
	        $y_3(0)=$&3\\
        	\hline
        	\end{tabular}
        \hspace{1mm}
		\begin{minipage}[c]{0.5\textwidth}
      		 \center{
		 \fbox{
		 \includegraphics[height=\textwidth]{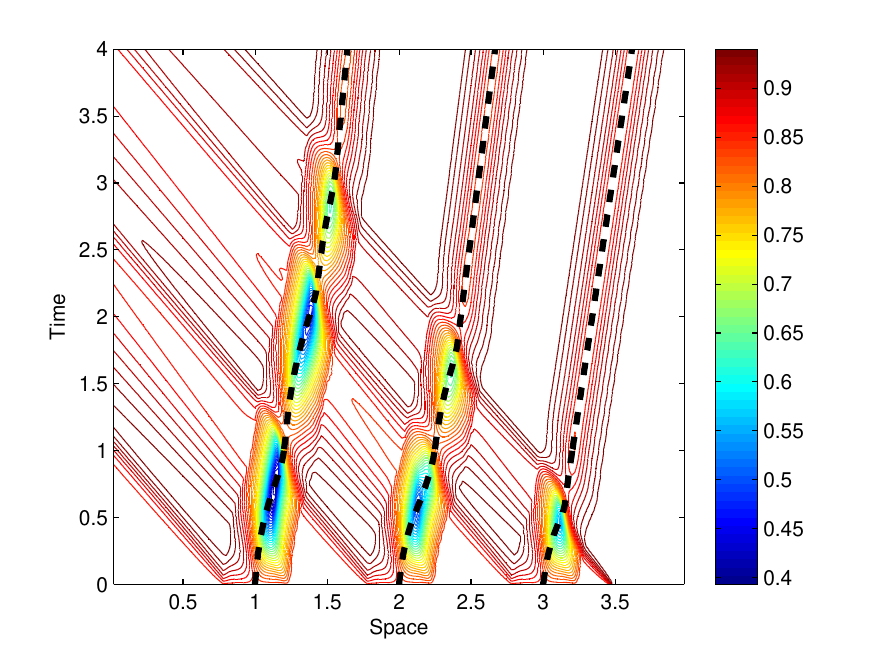}}}
		 \end{minipage}
	\caption{Three bottlenecks interacting with a shock.}
	\label{threeShock}
	\end{figure}
In Figure \ref{threeRar}, we set $\vmin_i=0.5$ and $\beta_i=0.25$ for $i=1,2,3$. The initial datum for the density is piecewise constant
with single discontinuity at $x=2.5$ and values $\rho_\ell=0.9$ and $\rho_r=0.1$. This gives rise to a rarefaction wave.
Each slow vehicle in the road is interacting with the rarefaction wave, but in different ways.
The second vehicle, starting from $y_2(0)=1.5$, is accelerating, but not as much as the first one, starting from $y_3(0)=2$.
This is due to differences in the rate of density decrease. The last one, starting from $y_1(0)=1$, has maximal velocity $w_{max}^1$
higher than the others. However, it is not capable of strong acceleration, since its initial distance w.r.t. the second is
equal to $\beta_1+\beta_2=0.5$, thus its speed is bounded by that of the second vehicle.
\\
In Figure \ref{threeShock}, we simulate an homogeneous case, i.e.\ all bottlenecks have the same characteristics.
We set $\vmin_i=0.5$ and $\beta_i=0.25$ for $i=1,2,3$. The initial datum is piecewise constant with $\rho_\ell=0.85$ for any $x\leq3.5$ and
$\rho_r=0.95$ for any $x>3.5$, thus generating a shock with negative speed $\lambda=-0.8$.
Initial data for bottlenecks are $y_{1,0}=1$, $y_{2,0}=2$ and $y_{3,0}=3$.
Starting from $t=0$, due to capacity dropping, the density behind each slow vehicle increases giving rise to a queue of cars.
This phenomenon generates further waves travelling and interacting with bottlenecks as well as the main shock.
When a bottleneck cross the shock, it enters a region with higher density, so that its velocity decreases.
On the other hand, during the interaction with each slow vehicle,
the speed of the shock decreases substantially, as it happens around $t=1$, $t=2.5$ and $t=4$.
	 	 

\section*{Conflict of interest}
All authors declare no conflicts of interest in this paper.

\end{document}